\documentclass{amsart}
\usepackage{a4wide,amsmath,amssymb}
\usepackage[toc,page]{appendix}
\usepackage{hyperref}
\usepackage{url}
\usepackage{mathtools}
\usepackage{xcolor}

\DeclarePairedDelimiter\floor{\lfloor}{\rfloor}

\newtheorem{thm}{Theorem}[section]
\newtheorem{lemma}[thm]{Lemma}
\newtheorem{prop}[thm]{Proposition}
\newtheorem{cor}[thm]{Corollary}
\theoremstyle{remark}
\newtheorem{rem}[thm]{Remark}

\newtheorem{ex}[thm]{Example}

\theoremstyle{definition}
\newtheorem{defn}[thm]{Definition}
\newtheoremstyle{Claim}{}{}{\itshape}{}{\itshape\bfseries}{:}{ }{#1}
\theoremstyle{Claim}

\newcommand{\R}{\mathbb{R}}

\newcommand{\eps}{\varepsilon}

\makeatletter
\@namedef{subjclassname@2020}{\textup{2020} Mathematics Subject Classification}
\makeatother

\theoremstyle{plain}
\def\sideremark#1{\ifvmode\leavevmode\fi\vadjust{% The remark
\vbox to0pt{\hbox to 0pt{\hskip\hsize\hskip1em% will appear only
\vbox{\hsize3cm\tiny\raggedright\pretolerance10000% on the side
\noindent #1\hfill}\hss}\vbox to8pt{\vfil}\vss}}}% in 3cm

\begin{document}

\title[]{High-order estimates for fully nonlinear equations under weak concavity assumptions}

\author{Alessandro Goffi}
\address{Dipartimento di Matematica ``Tullio Levi-Civita'', Universit\`a degli Studi di Padova, 
via Trieste 63, 35121 Padova (Italy)}
\curraddr{}
\email{alessandro.goffi@unipd.it}

\subjclass[]{}
 \thanks{%Partially supported by ....
 The author wishes to thank Martino Bardi and Cristina Giannotti for fruitful discussions and for providing many references on fully nonlinear equations, Kevin R. Payne for several pointers to references for Special Lagrangian equations and Xavier Ros-Oton for a comment on the regularity result for fully nonlinear elliptic equations in two variables. He is also grateful to the anonymous reviewer for the invaluable comments and a careful review, which meant a significant improvement of the first version of the manuscript. The author is member of the Gruppo Nazionale per l'Analisi Matematica, la Probabilit\`a e le loro Applicazioni (GNAMPA) of the Istituto Nazionale di Alta Matematica (INdAM). The author was partially supported by the INdAM-GNAMPA Project 2022 ``Propriet\`a quantitative e qualitative per EDP non lineari con termini di gradiente'', the INdAM-GNAMPA Project 2023 ``Problemi variazionali/nonvariazionali: interazione tra metodi integrali e principi del massimo''  and by the King Abdullah University of Science and Technology (KAUST) project CRG2021-4674 ``Mean-Field Games: models, theory and computational aspects".}
% % The authors were partially supported by the research project  ``Nonlinear Partial Differential Equations: Asymptotic Problems and Mean-Field Games" of the Fondazione CaRiPaRo. 
% }

\date{\today}
%\thanks{}

%%    The 2010 edition of the Mathematics Subject Classification is
%%    now available.  If you are citing a classification from the
%%    new scheme, use the following input coding instead.
%%\subjclass[2010]{Primary }
% \thanks{%Partially supported by ....
% The authors are members of the Gruppo Nazionale per l'Analisi Matematica, la Probabilit\`a e le loro Applicazioni (GNAMPA) of the Istituto Nazionale di Alta Matematica (INdAM)
% % The authors were partially supported by the research project  ``Nonlinear Partial Differential Equations: Asymptotic Problems and Mean-Field Games" of the Fondazione CaRiPaRo. 
% }

\date{\today}

\begin{abstract}
This paper studies a priori and regularity estimates of Evans-Krylov type in H\"older spaces for fully nonlinear uniformly elliptic and parabolic equations of second order when the operator fails to be concave or convex in the space of symmetric matrices. In particular, it is assumed that either the level sets are convex or the operator is concave, convex or close to a linear function near infinity. As a byproduct, these results imply polynomial Liouville theorems for entire solutions of elliptic equations and for ancient solutions to parabolic problems.
\end{abstract}

\maketitle


\tableofcontents
\section{Introduction}
In this paper, we study interior regularity properties in $C^{2,\alpha}$ and $C^{1,1}$ spaces for solutions $u$ of fully nonlinear elliptic and parabolic equations of the form
\begin{equation}\label{fullyellintro}
F(D^2u)=0\text{ in }B_1
\end{equation}
and
\begin{equation}\label{fullyparintro}
F(D^2u)-\partial_t u=0\text{ in }Q_1:=B_1\times(-1,0],
\end{equation}
under the assumption that $F:\mathcal{S}_n\to\R$, $\mathcal{S}_n$ being the space of $n\times n$ symmetric matrices, is uniformly elliptic/parabolic, meaning that
\begin{equation}\label{uell}
\lambda\|N\|\leq F(M+N)-F(M)\leq\Lambda\|N\|, \forall N\geq0,
\end{equation}
where $\|N\|=\sup_{|x|=1}|Nx|$ and $\Lambda\geq\lambda>0$.
The main peculiarity of the present article is that $F$ is not assumed concave or convex in $\mathcal{S}_n$. Instead, we assume only convexity of the superlevel (sublevel) sets, which means that $F$ is quasiconcave (resp. quasiconvex), or, alternatively, we impose some asymptotic concavity properties, namely $F=F(M)$ is concave (resp. convex or ``close to a linear function'') when $M$ is large. For parabolic equations \eqref{fullyparintro} in two-space variables, we do not impose any assumption other than \eqref{uell} as in the one-dimensional space setting \cite{Kruzkov}. We also discuss the case in which the ellipticity constants are close to each other, meaning that they satisfy a Cordes-type condition: this avoids to impose any concavity condition on $F$ and allows to consider any space dimension.\\
When $F$ is concave (resp. convex), the Evans-Krylov theorem \cite{Evans82, Krylov82} combined with $C^{1,1}$ estimates give an a priori $C^{2,\alpha}$ estimate in terms of the sup-norm of the unknown function, valid also for viscosity solutions, by the results in \cite{C91,CC}. The main underlying idea behind the $C^{2,\alpha}$ estimates is that any pure second derivative $u_{ee}$ is a subsolution (by the concavity of the operator) of a linear equation in nondivergence form, to which one can apply the Krylov-Safonov weak Harnack inequality. This and the functional relation given by the equation provide a two-side control on the oscillation of $D^2u$, see e.g. the proof of the a priori estimate given by L.C. Evans in \cite{BCESS} and the shorter one provided in \cite{CafSilv}.\\
 The second major chapter in the theory regards Calder\'on-Zygmund estimates proved in the paper \cite{C91}, see also \cite{Escauriaza} for a refinement, which gave a counterpart of the classical maximal $L^q$-regularity for linear elliptic equations. Extensions to the parabolic case of the aforementioned results were studied first in \cite{EvansLenhart}, then in \cite{Wang1,Wang2}.\\
When $F$ is assumed only uniformly elliptic or parabolic, solutions may fail to be smooth and the best regularity known under this sole assumption is the $C^{1,\alpha}$ regularity, cf. \cite{C91} and \cite{Trudinger88}. Roughly speaking, the main idea is that any directional derivative $u_e$ solves a linear uniformly elliptic equation in nondivergence form, and hence $u_e$ (or the first-order incremental quotient) is H\"older continuous by the Krylov-Safonov regularity theory. In this direction, N. Nadirashvili and S. Vl\u{a}du\c{t} \cite{NV,NTV,TeixeiraNotices} exhibited counterexamples in dimension $n\geq 5$ to the smoothness (of class $C^2$) of solutions, see also \cite{CaffarelliStefanelli,SilvestreAnnali} for related counterexamples in the parabolic case. To decide whether all solutions to uniformly elliptic equations in dimensions $n=3$ and $n=4$ are classical remains at this stage an open problem, whilst in dimension $n=2$ solutions of fully nonlinear uniformly elliptic equations are always classical, both in the elliptic \cite{Nirenberg,CaffarelliYuan,FRRO} and in the parabolic case \cite{Andrews,SilvestreAnnali}, see Theorem \ref{2dpar} for another proof of the result for viscosity solutions. We refer to \cite{FRRO} for a complete account on the theory. Other than this progress, we mention that $C^{2,\alpha}$ estimates can be achieved when the equation is close to be linear: this can be done assuming the uniform ellipticity combined with Cordes-Landis type conditions on the ellipticity constants (i.e. when $\frac{\Lambda}{\lambda}\leq 1+c(n)$), see Theorem 6.5 in \cite{SilvestreNotes} and \cite{BW,HuangAIHP,Niu}. This result however is still unknown in the parabolic setting. Some other $C^{2,\alpha}$ estimates have been obtained by compactness arguments in \cite{HuangLio} under the validity of a polynomial Liouville theorem, or when $\|u\|_{L^\infty}$ is small (the so-called case of flat solutions), cf. \cite{Savin} and {\cite{ASS}.\\
Nonetheless, even though in the uniformly elliptic setting solutions may fail to be smooth, some low integrability of $D^2u$ can be expected even without further assumption than \eqref{uell}: this is the object of the Evans-Lin $W^{2,\epsilon}$ estimates \cite{EvansTAMS,Lin}, see also Proposition 7.4 in \cite{CC} or \cite{C91} for a complete account. Moreover, recent works addressed high regularity estimates for semisolutions of elliptic equations under one-side geometric requirements, cf. \cite{BM,BFM}.\\
 
A significant problem in the theory is thus to determine some structural conditions on $F$ other than the uniform ellipticity (in between concavity or convexity and no hypotheses) that guarantee higher order $C^{2,\alpha}$ and $W^{2,p}$ estimates, cf. p. 574 in \cite{CCjmpa}. Important results in this direction can be found in \cite{CaffarelliYuan,CCjmpa,YuanAIHP,Collins16,Pingali} for fully nonlinear uniformly elliptic equations having a special structure in the context of $C^{2,\alpha}$ estimates. See also \cite{KrylovCCM,KrylovCPDE,DongKrylov} for Calder\'on-Zygmund estimates under relaxed convexity conditions. Recent works have been devoted to obtain high-order estimates in $C^{1,\alpha}$ and $W^{2,p}$ through the so-called recession operator, where $F$ is assumed to be concave or convex only at the ends of $\mathcal{S}_n$. This is done for instance in \cite{SilvestreTeixeira,PimentelTeixeira,Lee,LeeOk} and the references therein. We also refer to the works by N.V. Krylov et al, cf. \cite{KrylovBookNew}, for related assumptions and more references.\\
Our main results are inspired from an observation made by B. Andrews in \cite{Andrews}, who pointed out that the convexity of the level sets of $F$ as a function of $M$ is sufficient to derive a priori Evans-Krylov estimates for fully nonlinear equations (see also p. 575 of \cite{CCjmpa}). This slightly weakens the requirement of the Evans-Krylov theorem. In particular, in Theorem 6 of \cite{Andrews} an Evans-Krylov estimate was proved for fully nonlinear uniformly parabolic equations under the assumption that $F\in C^2$ and
\[
F_{ij,kl}M_{ij}M_{kl}\leq0\text{ for all }M\in\mathcal{S}_n\text{ satisfying }F_{ij}M_{ij}=0, F_{ij}=\frac{\partial F(M)}{\partial m_{ij}}, F_{ij,kl}(M)=\frac{\partial^2 F(M)}{\partial m_{kl}\partial m_{ij}}.
\]
Here, starting from this remark, we prove an a priori Evans-Krylov estimate for fully nonlinear uniformly elliptic equations that are quasiconcave (resp. quasiconvex) in $D^2u$ by means of an increasing transformation that makes the operator concave (resp. convex). More precisely, if $u$ solves \eqref{fullyellintro}, then we find a function $\psi:\R\to\R$ increasing with $\psi(0)=0$ such that $u$ solves
\[
G(D^2u)=\psi(F(D^2u))=0,
\]
where $G$ is concave and uniformly elliptic with ellipticity constants depending on $\|D^2u\|_{C(B_1)}$. This implies the $C^{1,1}$ to $C^{2,\alpha}$ estimates in the stationary case by the aforementioned results due to Evans and Krylov. Then, following Section 9 in \cite{CC}, by the Bernstein method we prove a priori $C^{1,1}$ estimates providing a control on the size of the $C^{2,\alpha}$ norm in terms of the sup-norm of the unknown, i.e.
\[
\|u\|_{C^{2,\alpha}(B_\frac12)}\leq C(n,\lambda,\Lambda,\|u\|_{C(B_1)},|F(0)|).
\]
This is done in Section \ref{sec;aprioriell} and complements the results of \cite{Andrews} in the elliptic setting.\\
The parabolic counterpart of the above a priori $C^{2,\alpha}$ estimate with respect to the parabolic distance for solutions to \eqref{fullyparintro} is more delicate, see Section \ref{sec;parapriori}. Indeed, it is worth remarking, as already done in \cite{Andrews}, that the previous transformation does not allow to convert the quasiconcave parabolic equation into a concave fully nonlinear parabolic PDE. Indeed, \eqref{fullyparintro} is not the same as the equation
\[
\psi(F(D^2u))-\partial_tu=0.
\]
To do this, we exploit again an idea of B. Andrews \cite{Andrews}, see also the earlier paper \cite{GutierrezHuang} and the more recent \cite{SilvestreAnnali}, and prove the estimate by regarding \eqref{fullyparintro} as an elliptic problem with a H\"older right-hand side. This step can be performed once one knows that $\partial_t u$ is space-time H\"older continuous. Then, an interpolation argument and the elliptic result provide the full a priori estimate of the form
\[
\|\partial_t u\|_{C^{\alpha,\alpha/2}(Q_\frac12)}+\|D^2u\|_{C^{\alpha,\alpha/2}(Q_\frac12)}\leq C(n,\lambda,\Lambda,\|u\|_{C(Q_1)},|F(0)|).
\]
We also prove regularity estimates for viscosity solutions in Sections \ref{sec;regell} and \ref{sec;regpar}. To do this, in the elliptic case we follow the classical route by \cite{CC} and prove that for quasiconcave operators the convexity of superlevel sets is enough to conclude that the second order differential quotients are still subsolutions (in the viscosity sense) of a Pucci's extremal equation. In the parabolic case, we again regard the equation as elliptic and exploit the previous stationary result. Then, one is only left to prove the time-H\"older continuity of second derivatives. The latter follows either adapting the two-dimensional arguments from \cite{Andrews,SilvestreAnnali} or via an interpolation argument. For equations in two-space variables and for those satisfying Cordes conditions, such an approach provides space-time $C^{2,\alpha}$ estimates without any additional condition on $F$ other than \eqref{uell}, cf. Section \ref{sec;2dpar} and Section \ref{sec;Cordes}.\\

We conclude the paper with $C^{2,1}$ estimates for fully nonlinear parabolic models that are concave/convex or close to a hyperplane at infinity. We mention that a priori $C^{2,1}$ estimates have been partially investigated in \cite{Kovats} in the concave case, where a one-side second derivative estimate is proved. Here, we address this issue by the (parabolic) Bernstein method, following \cite{Landis}. This analysis extends a work by L. Caffarelli and Y. Yuan \cite{CaffarelliYuan}, and also completes the result in the more classical concave setting. \\
Our results are also complemented with the derivation of $C^{2,\alpha}$ and $W^{2,p}$ estimates for equations with right-hand sides depending on $x$ using a perturbation argument due to L. Caffarelli \cite{C91,CC}. Here the idea is to import the regularity properties from those valid for the corresponding homogeneous equation with constant coefficients. \\
We also discuss some polynomial Liouville theorems for entire solutions of elliptic equations and ancient solutions of parabolic equations, as it is done for the Laplace and the Heat equation, cf. \cite{ColdingMinicozzi,FRRO}. We believe it would be worth investigating such issues in the context of Riemannian manifolds, where curvature conditions play a crucial role.
\par\bigskip
\paragraph{\textit{Notation}} 
We denote by $B_r(x)$ the ball of center $x$ and radius $r$. When $x_0=0$ we simply write $B_r$.\\
$Q_r(x,t)$ is the parabolic cylinder $B_r(x)\times(t-r^2,t)$, and we write $Q_r$ when $(x,t)=(0,0)$.\\
We denote, given $\Omega$ an open set, by $\partial_{\mathrm{par}}(\Omega\times(a,b))$ the parabolic boundary of a set.\\
We denote by $|u|_{0;\Omega}$ the sup-norm of $u$, i.e. $\|u\|_{C(\Omega)}$ (both in the elliptic and the parabolic case).\\
$u_e$ and $u_{ee}$ will denote the first- and second derivatives with respect to a unitary direction $e\in \R^n$.\\
$\mathrm{osc}_{\Omega}u$ denotes the oscillation of $u$ on $\Omega$\\
Let $Q\subset\Omega\times(0,T)$ and $\alpha\in(0,1)$. We denote by $d((x,t),(y,s))=|x-y|+|t-s|^\frac12$ the parabolic distance and
\begin{itemize}
\item $C^{\alpha,\alpha/2}(Q)$, $\alpha\in(0,1]$ the space of those $u:Q\to\R$ such that
\[
\|u\|_{C^{\alpha,\alpha/2}(Q)}:=|u|_{0;Q}+[u]_{C^{\alpha,\alpha/2}(Q)}=\|u\|_{L^\infty(Q)}+\sup_{\substack{(x,t),(y,s)\in Q,\\ (x,t)\neq (y,s)}}\frac{|u(x,t)-u(y,s)|}{d^\alpha((x,t),(y,s))};
\]
\item $C^{1+\alpha,(1+\alpha)/2}(Q)$ the space of those $u$ whose spatial gradient exists classically and endowed with the norm
\[
\|u\|_{C^{1+\alpha,(1+\alpha)/2}(Q)}:=|u|_{0;Q}+|Du|_{0;Q}+\sup_{\substack{(x,t),(y,s)\in Q,\\ (x,t)\neq (y,s)}}\frac{|u(x,t)-u(y,s)|}{d^{1+\alpha}((x,t),(y,s))}.
\]
In particular, any  $u\in C^{1+\alpha,(1+\alpha)/2}(Q)$ is such that each component of $Du$ belongs to $C^{\alpha,\alpha/2}(Q)$ and $u$ is H\"older continuous with exponent $(1+\alpha)/2$ in the time variable;
\item $C^{2+\alpha,1+\alpha/2}(Q)$ the space of functions $u$ such that
\[
\|u\|_{C^{2+\alpha,1+\alpha/2}(Q)}:=|u|_{0;Q}+\sum_{i=1}^n\|u_{x_i}\|_{C^{1+\alpha,(1+\alpha)/2}(Q)}+\|\partial_t u\|_{C^{\alpha,\alpha/2}(Q)}
\]
This can also be expressed by saying that $D^2u$ belongs to $C^{\alpha,\alpha/2}(Q)$ and $\partial_t u$ belongs to $C^{\alpha,\alpha/2}(Q)$, and can be obtained by means of Remark 8.8.7 in \cite{KrylovBookHolder}. In this case we can consider the space $C^{2+\alpha,1+\alpha/2}(Q)$ equipped with the norm
\[
\|u\|_{C^{2+\alpha,1+\alpha/2}(Q)}:=|u|_{0;Q}+|Du|_{0;Q}+\|D^2u\|_{C^{\alpha,\alpha/2}(Q)}+\|\partial_t u\|_{C^{\alpha,\alpha/2}(Q)}.
\]
 For more properties on these spaces we refer to \cite{KrylovBookHolder}. Furthermore, we denote by
\[
\|u\|_{C^{2,1}(Q)}:=\sum_{2i+j\leq 2}|\partial_t^iD_x^j u|_{0;Q}.
\]
\end{itemize}
Throughout this paper the ellipticity constants $\lambda,\Lambda$ are fixed, and a constant will be called universal when depends on $\lambda,\Lambda$ and the dimension $n$. $C$ will denote a generic positive constant depending on the data, that may differ in each inequality.\\
For $F=F(M)$, $M\in\mathcal{S}_n$, we denote by $F_{ij}(M)=\frac{\partial F(M)}{\partial m_{ij}}$ and $F_{ij,kl}(M)=\frac{\partial^2 F(M)}{\partial m_{kl}\partial m_{ij}}$.\\
When \eqref{uell} holds for the time-dependent problem \eqref{fullyparintro} we call $F$ uniformly parabolic.\\
From now on, the summation convention over repeating indices is implicitly understood.

\section{Preliminary definitions}
\subsection{Concavity and related notions}
We collect some tools from convex analysis, referring for more details to Chapter 3 of \cite{Boyd}. We recall that a function $u:\overline{\Omega}\to\R$ is called quasiconcave if its domain and all superlevel sets \[S^\alpha(u)=\{x:u(x)\geq \alpha\},\ \alpha\in\R\text{ are convex.}\] In particular, a function is quasiconcave if and only if its domain is convex and the following Jensen's inequality holds
\[
u(\theta x+(1-\theta)y)\geq \min\{u(x),u(y)\}\ ,\forall x,y\in\overline{\Omega},\theta\in[0,1].
\]
Similarly, a function is quasiconvex if $-u$ is quasiconcave, i.e. the sublevel sets are convex. For a quasiconvex $u$ it holds
\[
u(\theta x+(1-\theta)y)\leq \max\{u(x),u(y)\}\ ,\forall x,y\in\overline{\Omega},\theta\in[0,1].
\]
Recall that the superlevel sets of a concave function are convex, but the converse does not hold: for instance, the function $e^{x}$ is not concave on the real line (it is strictly convex), but the superlevel sets are convex. In particular, a concave function is also quasiconcave, but the converse does not hold. In fact, it can be proved that the function $u(x_1,x_2)=x_1x_2$ on $\R^2_+$ is neither concave nor convex, but it is quasiconcave on its domain. \\
If $u:\R^n\to\R$ is differentiable, then $u$ is quasiconcave if and only if the domain of $u$ is convex and for all $x,y\in \overline{\Omega}$ we have
\[
u(y)\geq u(x)\implies Du(x)\cdot (y-x)\geq0.
\]
Moreover, when $u:\overline{\Omega}\to\R$ is twice differentiable, if $u$ is quasiconcave then for all $x\in\overline{\Omega}$ and all $y\in\R^n$ one has
\begin{equation}\label{soc}
y\cdot Du=0\implies D^2uy\cdot y\leq0.
\end{equation}
This implies that when $Du=0$ we have $D^2u\leq0$, while in the case $Du\neq0$ the above condition implies that $D^2u$ is negative semidefinite on the orthogonal complement $(Du)^\perp$ (having dimension $n-1$). The latter entails that $D^2u$ can have at most one positive eigenvalue.\\
We have the following second-order characterization of quasiconcave functions:
\begin{lemma}\label{lemmasoc}
A twice differentiable function $u:\overline{\Omega}\to\R$ is quasiconcave if and only if there exists $\sigma\in\R$ such that
\[
D^2u+\sigma Du\otimes Du\leq0,
\]
where $(x\otimes y)_{ij}=x_iy_j$, $x,y\in\R^n$.
\end{lemma}
Given a function $F\in \mathcal{S}_n\to\R$ with $F=F(M)$, this can be extended to a function defined on the space of real $n\times n$ matrices through the extension $F(M)=F(\frac12(M+M^T))$. Then $F$ is a real-valued function of the $n\times n$ variables $m_{ij}$. Moreover, $DF(M)\cdot N=F_{ij}(M)N_{ij}$, $M,N\in\mathcal{S}_n$. \\
For a uniformly elliptic functional $F=F(M)$ satisfying \eqref{uell} (not necessarily of class $C^1$) we say that it is quasiconcave (resp. quasiconvex) if it is a quasiconcave (quasiconvex) function of $M\in\mathcal{S}_n$, namely for all $M_1,M_2\in\mathcal{S}_n$ and $\theta\in[0,1]$ we have
\[
F((1-\theta)M_1+\theta M_2)\geq \min\left\{F(M_1),F(M_2)\right\}\ (\leq \max\left\{F(M_1),F(M_2)\right\}).
\]
The second-order characterization \eqref{soc} then reads as follows when $F\in C^2$:
\[
F_{ij,kl}M_{ij}M_{kl}\leq0\text{ for all }M\in\mathcal{S}_n\text{ satisfying }F_{ij}M_{ij}=0.
\]
We point out that the latter condition is the one imposed in Theorem 6 of \cite{Andrews} to prove Evans-Krylov estimates.
Lemma \ref{lemmasoc} then implies the following result when specified to functions over symmetric matrices:
\begin{lemma}\label{condsigma}
A twice differentiable function $F:\mathcal{S}_n\to\R$, $F=F(M)$, is quasiconcave in the space of symmetric matrices if and only if it is quasiconcave as a function of $n\times n$ variables. Moreover, it is quasiconcave if and only if there exists $\sigma\in\R$ such that for all $M\in\mathcal{S}_n$
\[
F_{ij,kl}M_{ij}M_{kl}+\sigma F_{ij}F_{kl}M_{ij}M_{kl}\leq0.
\]
\end{lemma}
\begin{ex}
Consider the operator
\[
F(M):=\sum_{i=1}^n\arctan\lambda_i(M)
\]
and the fully nonlinear equation
\begin{equation}\label{sle}
\sum_{i=1}^n\arctan\lambda_i(D^2u)=f(x)
\end{equation}
with $|f(x)|\geq \frac{\pi}{2}(n-2)$. This is known as Special Lagrangian equation with supercritical phase. It was proved in Lemma 8.1 of \cite{CNSacta}, see also Lemma 2.1 in \cite{YuanPAMS},\cite{HL} or Lemma 2.1 in \cite{Collins} that if $F$ is seen as a function of the eigenvalues, i.e. $F=f(\lambda)=f(\lambda_1,...,\lambda_n)$, then
\[
\{\lambda\in\R^n:\sum_i\arctan\lambda_i\geq \frac{\pi}{2}(n-2)\}
\]
is a convex set. Therefore, the set
\[
\Sigma:=\{M\in\mathcal{S}_n: F(M)\geq \alpha\}, \alpha\geq \frac{\pi}{2}(n-2)
\]
is convex, cf. Corollary 2.3 in \cite{Collins}, which implies that the operator $F$ in \eqref{sle} is quasiconcave provided that $|f(x)|\geq \frac{\pi}{2}(n-2)$.
\end{ex}

\subsection{Viscosity solutions of fully nonlinear equations}
We recall the following notion of viscosity solution:
\begin{defn}
A continuous function $u$ in $\Omega$ is a viscosity subsolution (resp. supersolution) of \eqref{fullyellintro} in $\Omega$ when the following conditions are satisfied:\\
If $x_0\in\Omega$, $\varphi\in C^2(\Omega)$ and $u-\varphi$ has a local maximum at $x_0$, then
\[
F(D^2\varphi(x_0))\geq 0
\]
(resp. $u-\varphi$ has a local minimum at $x_0$, then $F(D^2\varphi(x_0))\leq 0$). $u$ is a viscosity solution if it is both a subsolution and a supersolution in the viscosity sense.
\end{defn}
By Proposition 2.4 in \cite{CC}, it is equivalent to replace $\varphi$ with a paraboloid. Similar definitions hold for the parabolic equation \eqref{fullyparintro}, cf. \cite{Wang1,Wang2,ImbertSilvestre}.\\
We now introduce two classes of solutions that are important to prove $C^{1,1}$ estimates.
\begin{defn}
Let $F$ be a continuous function in $\Omega$ and $\lambda\leq\Lambda$ two positive constants. We denote by $\underline{\mathcal{S}}(\lambda,\Lambda,f)$ the space of continuous functions $u$ in $\Omega$ that are subsolutions of the Pucci's maximal equation
\[
\mathcal{M}^+_{\lambda,\Lambda}(D^2u)= f(x)\text{ in the viscosity sense in $\Omega$},
\]
where
\[
\mathcal{M}^+_{\lambda,\Lambda}(M)=\sup_{\lambda I_n\leq A\leq \Lambda I_n}\mathrm{Tr}(AM)
\]
is the Pucci's maximal operator. Similarly, $\overline{\mathcal{S}}(\lambda,\Lambda,f)$ the space of continuous functions $u$ in $\Omega$ that are supersolutions to the Pucci's minimal equation
\[
\mathcal{M}^-_{\lambda,\Lambda}(D^2u)= f(x)\text{ in the viscosity sense in $\Omega$},
\]
where
\[
\mathcal{M}^-_{\lambda,\Lambda}(M)=\inf_{\lambda I_n\leq A\leq \Lambda I_n}\mathrm{Tr}(AM)
\]
is the Pucci's minimal operator.
\end{defn}
Similar definitions continue to hold for parabolic equations, cf. \cite{Wang1,Wang2}.
\section{A priori estimates for quasiconcave elliptic equations}\label{sec;aprioriell}
\subsection{From $C^{1,1}$ to $C^{2,\alpha}$ estimates: the case of smooth operators (and smooth solutions)}
In this section we prove an a priori $C^{2,\alpha}$ estimate depending on the $C^{1,1}$ norm of the unknown function for classical solutions to the Hessian equations $F(D^2u)=0$ under the assumption that $F\in C^2$ and it is a quasiconcave function in the space of symmetric matrices. This exploits a suitable transformation that allows to inherit this level of regularity from the Evans-Krylov theorem for concave functionals. This was already observed by B. Andrews in \cite{Andrews} for the case of quasilinear functions $F$, i.e. when $F$ is both quasiconcave and quasiconvex, meaning that its level sets (and not only the superlevel sets) are convex. A similar increasing transformation has been recently used even in the setting of Special Lagrangian equations with supercritical phases, cf. Corollary 2.3 in \cite{Collins}.
\begin{thm}\label{EvansKrylov}
Let $F\in C^2$ be quasiconcave and uniformly elliptic, i.e.
\[
\lambda|\xi|^2\leq F_{ij}\xi_i\xi_j\leq\Lambda|\xi|^2,
\]
with $F(0)=0$. Let
\[
u\in C^2(B_1)\text{ satisfy }F(D^2u)=0\text{ in }B_1.
\]
Then we have the a priori estimate
\[
\|u\|_{C^{2,\alpha}(\overline{B}_\frac12)}\leq C,
\]
where $\alpha\in(0,1)$ and $C$ are universal constants, i.e. depend on $n,\lambda,\Lambda$, but $C$ depends also on $\|u\|_{C^{1,1}(\overline{B}_\frac34)}$.
\end{thm}
\begin{proof}
We prove that there exists a function $\psi:\R\to\R$ with $\psi(0)=0$ and $\psi\in C^2$ such that $\psi(F):\mathcal{S}_n\to\R$ is concave and $u$ solves $\psi(F(D^2u))=0$ in $B_1$. We heavily use that $F\in C^2$ to exploit the second order characterization of quasiconcave functions. Indeed, if $F$ is quasiconcave we have that there exists $\sigma>0$ such that for any $M\in\mathcal{S}_n$ it holds
\[
F_{ij,kl}M_{ij}M_{kl}\leq\sigma F_{ij}F_{kl}M_{ij}M_{kl}.
\]
We define for $k>0$ to be selected the function $\psi:\R\to\R$ as
\[
\psi(r)=-e^{-k r}+1.
\]
It is immediate to see that $\psi$ is increasing and $\psi(0)=0$ if and only if $r=0$. We set
\[
G(M)=\psi(F(M)).
\]
We first observe that if $u$ solves $F(D^2u)=0$ with $F(0)=0$, then it solves $G(D^2u)=0$ in the same domain. We now check that $F$ is uniformly elliptic with ellipticity constants depending on $D^2u$ (and hence on $\|u\|_{C^{1,1}}$). Indeed
\[
G_{ij}=\psi'(F)F_{ij};
\]
\[
G_{ij,kl}=\psi''(F)F_{ij}F_{kl}+\psi'(F)F_{ij,kl}.
\]
Since
\[
\psi'(r)=k e^{-k r}\text{ and }\psi''(r)=-k^2e^{-k r}
\]
we get
\[
G_{ij,kl}=k e^{-k F}\left(F_{ij,kl}-k F_{ij}F_{kl}\right)
\]
and hence for all $M\in\mathcal{S}_n$
\[
G_{ij,kl}M_{ij}M_{kl}=k e^{-k F}\left(F_{ij,kl}M_{ij}M_{kl}-k F_{ij}F_{kl}M_{ij}M_{kl}\right).
\]
Then, if we choose $k\geq\max\{\sigma,0\}$ we get 
\begin{multline*}
G_{ij,kl}M_{ij}M_{kl}=k e^{-k F}\left(F_{ij,kl}M_{ij}M_{kl}-k F_{ij}F_{kl}M_{ij}M_{kl}\right)\\
\leq k e^{-k F}\left(F_{ij,kl}M_{ij}M_{kl}-\sigma F_{ij}F_{kl}M_{ij}M_{kl}\right)\leq0.
\end{multline*}

This implies that $G\in C^2$ is concave. Moreover
\[
G_{ij,kl}\xi_i\xi_j=\psi'(F)F_{ij}\xi_i\xi_j\text{ for all }\xi\in\R^n.
\]
Since $F$ is uniformly elliptic with ellipticity constants $\lambda,\Lambda$, then $G$ is uniformly elliptic with ellipticity constants $\lambda\psi'(F)$ and $\Lambda\psi'(F)$, which in turn depend on $|D^2u|_{0;B_{3/4}}$. We are then in position to apply the classical result of L.C. Evans \cite{Evans82}, see also Theorem 6.1 in \cite{CC} and \cite{KrylovBookNew}, to the $C^2$ solution of $G(D^2u)=0$ in $B_1$, where $G$ is a concave operator, and conclude the desired estimate.
\end{proof}

%\begin{rem}
%Observe that the Evans-Krylov estimate applies to the Special Lagrangian equation
%\[
%\sum_{i=1}^n\arctan\lambda_i(D^2u)=h(x)\text{ in }\Omega
%\]
%with $h(x)\geq \frac{\pi}{2}(n-2)$. Indeed, in this case one can show that the operator
%\[
%G(D^2u)=-e^{-kF(D^2u)}
%\]
%is concave for $k$ large enough depending on $\delta=\inf_\Omega h-\frac{\pi}{2}(n-2)$. This follows from \cite{YuanProcAms}, see also \cite{CPWcalcvar}.
%\end{rem}
\subsection{A priori $C^{1,1}$ estimates through the Bernstein method and some consequences}
We first prove a key result, as in \cite{CC}, which says that $u_e$ solves a linear equation in nondivergence form, while $u_{ee}$ is a subsolution to a linear equation, the second one being valid under quasiconcavity assumptions.
\begin{lemma}\label{Lue}
Let $F=F(M):\mathcal{S}_n\to\R$ be quasiconcave, $F\in C^2$ and such that $F(0)=0$. Consider the Hessian equation
\[
F(D^2u(x))=0\text{ in }\Omega.
\]
Then, setting $\mathcal{L}:=a_{ij}(x)\partial_{ij}=F_{ij}(D^2u(x))\partial_{ij}$ with $\lambda|\xi|^2\leq a_{ij}(x)\xi_i\xi_j\leq\Lambda|\xi|^2$, the following statements hold:
%\[
%\mathcal{L}v=a_{ij}(x)\partial_{ij}v=F_{ij}(D^2u(x))\partial_{ij}v
%\]
\begin{itemize}
\item[(i)] If $u\in C^3(\Omega)$ then
 \[
\mathcal{L}u\leq0\text{ and }\mathcal{L}u_e=0\text{ in }\Omega.
\]
\item[(ii)] If $u\in C^4(\Omega)$ then
\[
\mathcal{L}u_{ee}\geq0\text{ in }\Omega.
\]
\end{itemize}
%Let now $F$ be homogeneous of degree $k$ and $C^1$.
%\begin{itemize}
%\item[(iii)] Let $u\in C^3(\Omega)$ be such that $F(D^2u)=0$ in $\Omega$
%\[
%\mathcal{L}u=0\text{ and }\mathcal{L}u_e=0\text{ in }\Omega\]
%\item[(iv)] Moreover, if $k\geq 2$, we have
%\[
%\mathcal{L}u_{ee}=0\text{ in }\Omega.
%\]
%and if $k\geq m$, $m\in\N$, we have, denoting by $u^{(m)}_{e}$ the m-th derivative respect to the unitary direction $e$,
%\[
%\mathcal{L}u^{(m)}_{e}=0\text{ in }\Omega.
%\]
%\end{itemize}
\end{lemma}
\begin{proof}
Let $\psi(t)=F((1-t)D^2u)$. We have $\psi(0)=0$, $\psi(1)=F(0)=0$. Since $F$ is quasiconcave, it follows that $\psi\geq0$ in $[0,1]$. Therefore, by the first-order conditions, being $F\in C^1$ and $\psi(1)=\psi(0)$, we get
\[
0\leq \psi'(0)\cdot (1-0)=F_{ij}(D^2u(x))(-\partial_{ij}u(x))=-\mathcal{L}u
\]
Differentiating the equation $F(D^2u(x))=0$ we get $\mathcal{L}u_e=0$. This proves (i).\\
To prove (ii) we use that for a quasiconcave $F$ by Lemma \ref{condsigma} there exists a constant $\sigma\in\R$ such that for all $M\in\mathcal{S}_n$
\[
F_{ij,kl}(D^2u(x))M_{ij}M_{kl}\leq \sigma F_{ij}(D^2u(x))F_{kl}(D^2u(x))M_{ij}M_{kl}.
\]
Thus, we differentiate once more the equation to find, using that $\mathcal{L}u_e=0$
\[
0=\mathcal{L}u_{ee}+F_{ij,kl}(D^2u(x))(u_e)_{ij}(u_e)_{kl}\leq \mathcal{L}u_{ee}+\sigma F_{ij}(D^2u(x))F_{kl}(D^2u(x))(u_e)_{ij}(u_e)_{kl}=\mathcal{L}u_{ee}.
\]
Alternatively, instead of using Lemma \ref{condsigma}, one can exploit the characterization \eqref{soc} for quasiconcave functions to conclude
\[
DF(D^2u)\cdot D^2u_e=F_{ij}(D^2u)(u_{e})_{ij}=\mathcal{L}u_{e} =0\implies F_{ij,kl}(D^2u(x))(u_e)_{ij}(u_e)_{kl}\leq0.
\]
%To prove (iii), we note that if $F$ is homogeneous of degree $k$, by the Euler theorem for homogeneous (differentiable) functions it follows that
%\[
%F_{ij}(D^2u)\partial_{ij}u=kF(D^2u).
%\]
%Therefore, since $u$ solves $F(D^2u)=0$, also $kF(D^2u)=0$ and hence $\mathcal{L}u_e=0$ by differentiating the equation. By further applying the Euler theorem to the second derivatives of $F$ we get
%\[
%0=\mathcal{L}u_{ee}+F_{ij,kl}(D^2u(x))(u_e)_{ij}(u_e)_{kl}=\mathcal{L}u_{ee}+(\sigma-1)\mathcal{L}u_e=\mathcal{L}u_{ee}
%\]
%The last item (iv) follows by an iteration.
\end{proof}
\begin{rem}
The previous lemma continues to hold when $u\in C^2$ using difference quotients instead of derivatives \cite{CDV}.
\end{rem}
We also need the following result that is a consequence of the uniform ellipticity (no other assumptions on $F$ are required), cf. Lemma 6.4 in \cite{CC}.
\begin{lemma}\label{twoside}
If $F(M_1)=F(M_2)=0$ and $F$ is uniformly elliptic, then
\[
\frac{\lambda}{\Lambda+\lambda}\|M_2-M_1\|\leq \|(M_2-M_1)^+\|=\sup_{e\in\R^n,|e|=1}(e^T(M_2-M_1)e)^+.
\]
\end{lemma}
We prove the following $C^{1,1}$ interior estimates for quasiconcave nonlinear $C^2$ functionals $F$
\begin{thm}\label{C11}
Let $F$ be quasiconcave, uniformly elliptic, $F\in C^1$ with $F(0)=0$. Let $u\in C^3(\overline{B}_1)$ be such that $F(D^2u)=0$ in $B_1$. Then
\[
|Du|_{0;B_{1/2}}\leq C\mathrm{osc}_{B_1}u.
\]
If, in addition, $F\in C^2$ and $u\in C^4$, we have the estimate
\[
|D^2u|_{0;B_{1/2}}\leq C\mathrm{osc}_{B_1}u.
\]
Here, $C$ is a universal constant, depending only on $n,\lambda,\Lambda$. 
\end{thm}
\begin{proof}
The proof is the same as that in \cite{CC}, we sketch it here for reader's convenience. Define, for $\delta>0$ to be chosen later, the auxiliary function
\[
z=\delta(M-u)^2+\eta^2|Du|^2\in C^2(\overline{B}_1)\ ,M=\sup_{\overline{B}_1}u.
\]
where $\eta$ is an appropriate cut-off function. One proves by the maximum principle that $\sup_{B_\frac12}z\leq \delta (\mathrm{osc}_{B_1}u)^2$ by observing that for $\delta$  chosen large enough and a positive constant $C>0$
\[
\mathcal{L}z\geq 2\delta\lambda|Du|^2- C|Du|^2\geq0.
\]
To prove the second derivative estimate one can consider, for a unitary direction $e\in \R^n$ and $\mu>0$ to be chosen later, the function
\[
h=\mu|Du|^2+\eta^2(u_{ee}^+)^2\in C^2(\overline{B}_1)\ ,
\]
where $v^+=\max\{v,0\}$ is the positive part of $v$. A similar calculation that requires to differentiate the equation twice and exploits $\mathcal{L}u_{x_k}=0$ together with $\mathcal{L}u_{ee}\geq0$ (see Lemma \ref{Lue}) leads now to
\[
\mathcal{L}h\geq0
\]
in the set $\{x\in B_1:u_{ee}(x)>0\}$ and  $\mu$ large. Therefore this leads to the one-side bound
\[
u_{ee}^+\leq C\text{ in }B_\frac12.
\]
Using the equation $F(D^2u)=0$ together with Lemma \ref{twoside} (applied with $M_2=D^2u$ and $M_1=0$, using also $F(0)=0$), we get a two-side bound on $D^2u$.
\end{proof}
\begin{rem}
An observation made in \cite{CDV} allows to weaken the requirement $u\in C^4$ by using difference quotients and work with viscosity solutions. The basic idea relies on considering $z$ of the following form
\[
z=\eta^2\left(\frac{u(x+he)-u(x)}{h}\right)^2+\delta \int_0^1 u^2(x+the)\,dt,
\]
where the full derivatives have been replaced by a difference quotient and the zeroth order term has been replaced with an average value of $u^2$ in the direction of $e$.
\end{rem}

\begin{cor}\label{lioell1}
Let $F$ be quasiconcave, uniformly elliptic, $F\in C^2$ with $F(0)=0$. Let $u$ be a classical solution of $F(D^2u)=0$ in $\R^n$. Assume that
\[
|u(x)|\leq C(1+|x|^\gamma)\ ,\gamma<2, x\in\R^n.
\]
Then, $u$ must be a polynomial of degree at most $\floor*{\gamma}=1$, i.e. $u$ is at most an affine function. In particular, if either $\gamma\in(0,1)$ (i.e. $u$ has sublinear growth) or $\gamma=0$ (i.e. $u$ is bounded), then $u$ must be a constant. 
\end{cor}
\begin{proof}
The proof follows from the gradient and Hessian bounds stated for solutions of $F(D^2u)=0$ in $B_R$.
\[
|Du|_{0;B_{R/2}}\leq \frac{C}{R}\mathrm{osc}_{B_R}u,
\]
\[
|D^2u|_{0;B_{R/2}}\leq \frac{C}{R^2}\mathrm{osc}_{B_R}u.
\]
This can be done adjusting the previous proof as in \cite{KovatsEll}. Alternatively, it is enough to observe that if $u$ solves $F(D^2u)=0$, then $v(x)=\frac{1}{\rho}u(\omega x)$ solves
\[
\frac{\omega^2}{\rho}F\left(\frac{\rho}{\omega^2}D^2v\right)=0,
\]
where the constant in front of $F$ allows to preserve the uniform ellipticity of the starting equation. Therefore, the estimates in $B_R$ follow by scaling.
\end{proof}
An immediate consequence of Theorem \ref{EvansKrylov} and the previous $C^{1,1}$ a priori estimates is a $C^{2,\alpha}$ a priori estimate. Note that this statement is much weaker than that of Theorem 6.6 in \cite{CC}, since it is not stated for viscosity solutions, although it is valid under slightly weaker assumptions on $F$ compared with \cite[Theorem 6.6]{CC} and \cite[Corollary 9.4]{CC} (i.e. we are assuming only convexity of the superlevel sets).
\begin{cor}
Let $F$ be quasiconcave and satisfy $F\in C^\infty$. Let $u\in C^2(B_1)$ be a solution of $F(D^2u)=0$ in $B_1$. Then $u\in C^\infty(B_1)$ and
\[
\|u\|_{C^{2,\alpha}(B_\frac12)}\leq C(n,\lambda,\Lambda,|u|_{0;B_1},|F(0)|).
\]
\end{cor}
\begin{proof}
The proof follows combining Theorem \ref{EvansKrylov} with the $C^{1,1}$ a priori estimate from Theorem \ref{C11}. The smoothness of $u\in C^\infty$ follows from the bootstrap argument in Proposition 9.1 of \cite{CC} since $F$ itself is smooth.
\end{proof}
The previous estimate leads to the following polynomial Liouville theorem, as in the case of the Laplace equation, cf. \cite{ColdingMinicozzi,FRRO}.
\begin{cor}\label{lioell}
Let $F$ be quasiconcave, uniformly elliptic, $F\in C^2$ with $F(0)=0$. Let $u$ be a classical solution of $F(D^2u)=0$ in $\R^n$. Assume that
\[
|D^2u|\leq C.
\]
Then, $u$ must be a polynomial of degree at most $2$. 
\end{cor}
\section{$C^{2,\alpha}$ regularity estimates for viscosity solutions of quasiconcave equations}\label{sec;regell}

In this section we prove that viscosity solutions to quasiconcave equations $F(D^2u)=0$ satisfy interior $C^{2,\alpha}$ estimates for some universal $\alpha\in(0,1)$ in terms of $|u|_{0;B_1}$. Our main result reads as follows, and extends Theorem 6.6 in \cite{CC} under weaker conditions on $F$.
\begin{thm}\label{mainell}
Let $F$ be quasiconcave, continuous (with no other smoothness assumptions) and uniformly elliptic, and $u$ be a continuous viscosity solution to $F(D^2u)=0$ in $B_1$. Then $u\in C^{2,\alpha}(\overline{B}_\frac12)$ and
\[
\|u\|_{C^{2,\alpha}(\overline{B}_\frac12)}\leq C,
\]
where $0<\alpha<1$ and $C$ are universal constants. $C$ here also depends on $|u|_{0;B_1}$ and $|F(0)|$.
\end{thm}
\begin{rem}
We can always reduce, for any uniformly elliptic $F$ (not necessarily concave or quasiconcave), to $F(0)=0$ in proving Theorem \ref{mainell}, cf. Remark 1 in Section 6.2 of \cite{CC}.
\end{rem}
To prove Theorem \ref{mainell} we start by proving $C^{1,1}$ regularity estimates for viscosity solutions. The main result to prove this level of regularity is the following
\begin{thm}\label{C11visc}
Let $F$ be uniformly elliptic, continuous and quasiconcave. Let $u\in C(B_1)$ be a viscosity solution to
\begin{equation}\label{eqell}
F(D^2u)=0\text{ in }B_1.
\end{equation}
Then, it satisfies the regularity estimate
\[
[u]_{C^{1,1}(B_\frac12)}\leq C(|u|_{0;B_1}+|F(0)|).
\]
\end{thm}
Following the path outlined in \cite{CC}, the proof of this result requires to show that second order differential quotients belong to $\underline{\mathcal{S}}(\lambda,\Lambda,0)$ when $F$ is quasiconcave. This is a consequence of the following two results (note that only the second one requires concavity assumptions on $F$)
\begin{thm}\label{u-v}
Let $F$ be uniformly elliptic, $u$ be viscosity subsolution to \eqref{eqell} in $\Omega$ and $v$ be a viscosity supersolution of \eqref{eqell} in $\Omega$. Then
\[
u-v\in\underline{\mathcal{S}}(\lambda/n,\Lambda)\text{ in }\Omega.
\] 
\end{thm}
\begin{proof}
This is proved in Theorem 5.3 of \cite{CC}. Alternatively, one can proceed using the doubling variable argument \cite{CIL,ImbertSilvestre} and exploit the uniform ellipticity.
\end{proof}
\begin{cor}
Let $u$ be a viscosity solution to \eqref{eqell}. Let $h>0$ and $e\in\R^n$ with $|e|=1$. Then
\[
u(x+he)-u(x)\in\mathcal{S}(\lambda/n,\Lambda)=\underline{\mathcal{S}}(\lambda/n,\Lambda,0)\cap \overline{\mathcal{S}}(\lambda/n,\Lambda,0)\text{ in }\Omega_h:=\{x\in\Omega:\mathrm{dist}(x,\partial\Omega)>h\}.
\]
\end{cor}
\begin{proof}
This is a consequence of Theorem \ref{u-v} and the fact that $v(x)=u(x+he)$ is a viscosity solution to $F(D^2u)=0$ in $\Omega_h$.
\end{proof}
The next result exploits quasiconcavity properties on $F$ and shows that any convex combination of two subsolutions is still a subsolution.
\begin{thm}\label{muuv}
Let $F$ be uniformly elliptic, continuous and quasiconcave, and let $u,v$ be viscosity subsolutions to \eqref{eqell}. Then
\[
w=\theta u +(1-\theta)v,\theta\in(0,1),
\]
is a viscosity subsolution to $F(D^2u)=0$ in $B_1$.
\end{thm}
\begin{proof}
The proof is easy when $u,v$ are classical subsolutions. Indeed, by the Jensen's inequality for quasiconcave functions it follows that
\[
F\left(D^2\left(\theta u+(1-\theta)v\right)\right)\geq \min\{F(D^2u),F(D^2v)\}\geq0.
\]
In the general case the proof is a consequence of the following fact: if $u^\eps$ denotes the sup-convolution of $u$ and $u_\eps$ the inf-convolution, and also $u$ is a viscosity subsolution to $F(D^2u)=0$, then for $\delta>0$ small enough $(u^{\eps+\delta})_\delta\in C^{1,1}$ and is a viscosity subsolution of the same equation, so that one can proceed as in the smooth setting, see e.g. Proposition 4.6 and Remark 4.8 of \cite{CKSS}. Alternatively, one can proceed following the same steps of Theorem 5.8 of \cite{CC}: it is enough to see that $w=\theta u^\eps +(1-\theta)v^\eps$, $\mu\in(0,1)$, is a viscosity subsolution to $F(D^2w)=0$. Let $p$ be a paraboloid touching the graph of $\theta u^\eps +(1-\theta)v^\eps$ from above at $x_0$. We prove that $F(D^2p)\geq0$. Define
\[
z(x)=p(x)+\delta|x-x_0|^2-\delta r^2-(\theta u^\eps +(1-\theta)v^\eps).
\]
A standard argument that uses the ABP maximum principle (using Lemma 3.5 in \cite{CC} and arguing as in Theorem 5.3 of \cite{CC}) implies that there exists $x_1\in B$ such that  $u^\eps,v^\eps$ and $z$ are punctually second order differentiable at $x_1$, and
\begin{equation}\label{test}
D^2\left(p+\delta|x-x_0|^2-(\theta u^\eps +(1-\theta)v^\eps)\right)(x_1)\geq0.
\end{equation}
Moreover, $F(D^2u^\eps)\geq0$ and $F(D^2v^\eps)\geq0$. Since $F$ is quasiconcave we have
\[
F\left(D^2\left(\theta u^\eps +(1-\theta)v^\eps\right)\right)\geq \min\{F(D^2u^\eps),F(D^2v^\eps)\}\geq0
\]
This implies by \eqref{test} and the uniform ellipticity (actually, degenerate ellipticity is enough) that
\[
F(D^2p+2\delta I_n)\geq0
\]
Letting $\delta\to0$, since $F$ is continuous, we get $F(D^2p)\geq0$.
\end{proof}
The next corollary shows what we already observed for smooth functionals  in Lemma \ref{Lue}: the pure second derivatives $u_{ee}$ (or better the second order differential quotients) are a subsolution to a linear equation in nondivergence form (resp. a Pucci's extremal equation).
\begin{cor}\label{quot}
Let $F$ be quasiconcave and $u$ be a continuous viscosity solution to \eqref{eqell}. Then 
\[
\Delta_{he}^2u=\frac{u(x+he)+u(x-he)-2u(x)}{h^2}\in\underline{\mathcal{S}}(\lambda/n,\Lambda,0)\text{ in }B_\frac12.
\]
In particular, if $u\in C^{2}(Q_1)$, then
\[
u_{ee}\in \underline{\mathcal{S}}(\lambda/n,\Lambda,0)
\]
\end{cor}
\begin{proof}
We write for $\theta=\frac12$
\[
\frac{u(x+he)+u(x-he)-2u(x)}{2}=\frac12(u(x+he)+u(x-he))-u(x),
\]
which is by Theorem \ref{u-v} the difference of a viscosity subsolution and a supersolution of $F(D^2w)=0$. The result follows then from Theorem \ref{muuv}, since $\frac12(u(x+he)+u(x-he))$ is a viscosity subsolution.
\end{proof}
\begin{proof}[Proof of Theorem \ref{C11visc}]
The result follows from Corollary \ref{quot} and using statement (2) in Theorem 4.8 (the local maximum principle) of \cite{CC}. We detail part of the proof following the one in Theorem 6.6 in \cite{CC}. To do this, one observes that when $F$ is quasiconcave there exists a supporting hyperplane (by above) to the graph of $F$ at $0\in\mathcal{S}_n$. Indeed, it is enough to apply the Hahn-Banach theorem to the open convex set
\[
\{(M,\xi)\in \mathcal{S}_n\times\R:F(M)>\xi\}
\]
(recall that $F$ is a quasiconcave function when the superlevel sets are convex). This implies that there exists a linear functional $L$ on the space of symmetric matrices, i.e. of the form $L(M)=\mathrm{Tr}(AM)$ for some constant coefficient matrix $A$ such that $\lambda I_n\leq A\leq\Lambda I_n$ (that we may assume to be the Laplacian by Lemma 6.1 in \cite{GT}) such that $L(0)=0$ and $L(M)\geq F(M)$ for any $M\in\mathcal{S}_n$.\\
Therefore, since $F(D^2u)=0$ in the viscosity sense in $B_1$ and $\Delta \varphi=L(D^2\varphi)\geq F(D^2\varphi)$ for any $\varphi\in C^2$, we immediately see that $u$ solves $\Delta u\geq0$ in the viscosity sense in $B_1$. Then, we can continue the proof with the program outlined in \cite{CC} to prove the $C^{1,1}$ regularity estimate, see also \cite{CCtop}, exploiting now Corollary \ref{quot} valid for quasiconcave operators instead of the more classical Corollary 5.10 in \cite{CC}.
\end{proof}
We now state the following extension of Theorem 6.1 in \cite{CC}.
\begin{thm}\label{EKvisc}
Let $F$ be quasiconcave and let
\[
u\in C^{1,1}(B_1)\text{ satisfy }F(D^2u)=0\text{ in }B_1.
\]
Then $u\in C^{2,\alpha}(B_\frac12)$ and
\[
\|u\|_{C^{2,\alpha}(B_\frac12)}\leq C,
\]
where $\alpha\in(0,1)$ and $C$ are universal constants. Here $C$ depends on $\|u\|_{C^{1,1}(B_{\frac34})}$.
\end{thm}
\begin{proof}
Having the $C^{1,1}$ estimates of Theorem \ref{C11visc} at our disposal, we can exploit the increasing exponential transformation of Theorem \ref{EvansKrylov}. In this case, one can check that if $u$ solves $F(D^2u)=0$ in the viscosity sense, then $u$ solves in the viscosity sense the concave equation $G(D^2u)=\psi(F(D^2u))=0$. Also, since $\psi$ is increasing and $u\in C^{1,1}$, then $G$ continues to be uniformly elliptic (where now the ellipticity constants do not depend on $\|u\|_{C^{1,1}}$ by Theorem \ref{C11visc}). Therefore, the result follows from the regularity estimate in Theorem 6.1 in \cite{CC} stated for viscosity solutions.
\end{proof}
We are now ready for the proof of the full $C^{2,\alpha}$ estimate for viscosity solutions
\begin{proof}[Proof of Theorem \ref{EKvisc}]
The proof follows combining Theorem \ref{C11visc} with Theorem \ref{EKvisc}.
\end{proof}
\begin{rem}
The same results in Corollaries \ref{lioell1} and \ref{lioell} hold for viscosity solutions using the regularity estimate in Theorem \ref{mainell}.
\end{rem}
\subsection{Estimates for more general equations by perturbation}\label{sec;pert}
As a byproduct of a perturbation argument of L. Caffarelli we get the following interior pointwise $C^{2,\alpha}$ estimate for equations with $C^\alpha$ right-hand side.
\begin{thm}
Let $F=F(x,M)$ be uniformly elliptic, quasiconcave, continuous in $x$ and such that $F(0,0)=f(0)=0$. Suppose that $\bar\alpha\in(0,1)$ is the constant of the $C^{2,\bar\alpha}$ estimate for the homogeneous equation of Theorem \ref{mainell}. Assume also that $\alpha\in(0,\bar\alpha)$, $r_0>0$, $C_1,C_2>0$,
\[
\left(\frac{1}{|B_r|}\int_{B_r}\tilde{\beta}^n\right)^\frac1n\leq C_1r_0^{-\alpha}r^\alpha,\ \left(\frac{1}{|B_r|}\int_{B_r}|f|^n\right)^\frac1n\leq C_2r_0^{-\alpha}r^\alpha,\ \forall r\leq r_0.
\]
where
\[
\tilde\beta(x)=\sup_{M\in\mathcal{S}_n}\frac{|F(x,M)-F(0,M)|}{\|M\|+1}
\]
and $u$ be a viscosity solution of 
\[
F(x,D^2u)=f(x)\text{ in }B_{r_0}.
\]
Then $u\in C^{2,\alpha}$ at the origin.
\end{thm}
\begin{proof}
The result is a consequence of Theorem 8.1 in \cite{CC} and Theorem \ref{mainell}.
\end{proof}
Finally, the previous $C^{1,1}$ regularity estimates for the frozen equation $F(x_0,D^2u)=0$ imply $W^{2,q}$ estimates for $x$-dependent equations of the form $F(x,D^2u)=f(x)\in L^q$, $q>n-\eps$, where $\eps>0$ is a universal parameter depending on $n,\Lambda/\lambda$, again as a consequence of a result of L. Caffarelli \cite{CC} and L. Escauriaza \cite{Escauriaza}, cf. Theorem 7.1 and the subsequent Remark 3 in \cite{CC}. This is done by an assumption that measures the oscillation of $x$ in the critical space $L^n$, see \eqref{beta}.
\begin{thm}\label{czest}
Let $u$ be a bounded viscosity solution to
\[
F(x,D^2u)=f(x)\text{ in }B_1.
\]
Assume that $F(x,0)=0$ in $B_1$, $F$ continuous in $x$ and uniformly elliptic in $M\in\mathcal{S}_n$  and that $F(x_0,M)$ is quasiconcave in $M\in\mathcal{S}_n$ for any $x_0\in B_1$. Let also $c_e$ be the constant such that for any $x_0\in B_1$
\[
\|w\|_{C^{1,1}(\overline{B}_\frac12)}\leq c_{Z}.
\]
Let also $q$ be such that $\bar{q}(n,\lambda,\Lambda)<q<\infty$ and $f\in L^q(B_1)$, $\bar{q}$ being the Escauriaza's exponent. Then there exist positive constants $\beta_0$ and $C$ depending on $n,\lambda,\Lambda,c_Z$ and $q$ such that if
\begin{equation}\label{beta}
\left(|B_r(x_0)|^{-1}\int_{B_r(x_0)}\beta(x,x_0)^n\,dx\right)^\frac1n\leq\beta_0,
\end{equation}
\[
\beta(x,x_0)=\sup_{M\in\mathcal{S}_n/\{0\}}\frac{|F(x,M)-F(x_0,M)|}{\|M\|},
\]
for any ball $B_r(x_0)\subset B_1$, then $u\in W^{2,q}(B_\frac12)$ and 
\[
\|u\|_{W^{2,q}(B_\frac12)}\leq C(\|u\|_{L^\infty(B_1)}+\|f\|_{L^q(B_1)}).
\]

\end{thm}
\begin{rem}
As pointed out in \cite{C91}, one can allow discontinuous dependence on the operator $F$ with respect to $x$ and in the right-hand side $f$ by considering $L^p$-viscosity solutions. Results in this direction can be found in the papers \cite{CCKS,CKSpar,Koike,Pimentel}.
\end{rem}
\begin{rem}
We believe that the first- and second derivatives bounds obtained by the Bernstein method in Section \ref{sec;aprioriell} hold even for more general equations with lower order terms as in \cite{TrudingerTAMS} under natural structural conditions on $F=F(x,u,Du,D^2u)$. These typically hold when the growth in the gradient variable is at most quadratic. In fact, it is well-known that $C^{1,\alpha}$ estimates for such general equations hold without concavity assumptions and under much weaker regularity conditions than \cite{TrudingerTAMS}, cf. \cite{Nornberg,Swiech97}. Once $C^{1,\alpha}$ regularity is established, higher $C^{2,\alpha}$ estimates can be achieved by bootstrapping and exploiting the results of this and the previous sections by regarding the lower order terms as a right-hand side $g(x)$ for an equation of the form
\[
G(x,D^2u)=g(x).
\]
\end{rem}
\section{A priori $C^{2,\alpha}$ estimates for parabolic quasiconcave equations}\label{sec;parapriori}
We start with an extension of Theorem 6 in \cite{Andrews} for parabolic Hessian equations. We give here a slightly different proof resorting to an interpolation inequality of \cite{LSU}. Note that the next result holds for more general operators than those in \cite{Andrews}, where it was assumed that $F$ is quasilinear (i.e. both quasiconcave and quasiconvex). However, as already pointed out in \cite{Andrews}, here we cannot directly use a monotone transformation as in the stationary case to extract information from a concave parabolic equation. Indeed, the equation $F(D^2u)-\partial_tu=0$ is not the same as $\psi(F(D^2u))-\partial_t u=0$ and we need some extra work.
\begin{thm}\label{EKpar}
Let $F\in C^2$ be quasiconcave and uniformly parabolic, i.e.
\[
\lambda|\xi|^2\leq F_{ij}\xi_i\xi_j\leq\Lambda|\xi|^2.
\]
Let
\[
u\in C^{4,2}(Q_1)\text{ be a solution of }F(D^2u)-\partial_t u=0\text{ in }Q_1=B_1\times (-1,0].
\]
Then $u\in C^{2+\alpha,1+\frac{\alpha}{2}}(\overline{Q}_\frac12)$ and
\[
\|u\|_{C^{2+\alpha,1+\frac{\alpha}{2}}(\overline{Q}_\frac12)}\leq C,
\]
where $\alpha\in(0,1)$ and $C$ are universal constants, with $C$ depending also on $\|u\|_{C^{2,1}(\overline{Q}_1)}$.
\end{thm}
\begin{proof}
The idea of the proof follows the observation of B. Andrews \cite{Andrews} in the two-dimensional case, see also the recent work \cite{SilvestreAnnali}. Here we simplify the argument using a space-time interpolation inequality. This was also used earlier by C. Guti\'errez and Q. Huang to prove a priori estimates for parabolic Monge-Amp\`ere equations \cite{GutierrezHuang}, cf. Theorem 2.3 therein. By the smoothness of $F$ in the range of $D^2u$ and differentiating in time, it is immediate to see that $\partial_t u$ solves the linear PDE
\[
F_{ij}(D^2u(x,t))\partial_{ij}(\partial_tu)-\partial_t (\partial_tu)=0.
\]
We are then in position to apply the parabolic Krylov-Safonov H\"older regularity result, cf. Theorem 7 in \cite{Andrews}, and obtain that $\partial_tu$ is space-time H\"older continuous with the following estimate
\[
[\partial_t u]_{C^{\gamma,\frac{\gamma}{2}}(Q_{\frac34})}\leq C|\partial_t u|_{0;Q_1}.
\]
To estimate $D^2u(x,t)$ in the $x$-variable we freeze the time variable and consider the parabolic problem as a stationary equation with $\partial_tu$ as a right-hand side. That is, $v(x)=u(x,t)$ solves
\[
G(D^2v(x))=\phi(x),\ G(D^2v(x))=F(D^2u(x,t)),\ \phi(x)=\partial_tu(x,t).
\]
By the elliptic result in Theorem \ref{mainell} (combined with the fact that $\partial_t u$ is H\"older continuous in space and it is thought as a H\"older continuous right-hand side of an elliptic problem) on each time slice we have $\|D^2u(\cdot,t)\|_{C^\iota(B_{1/2})}\leq C$ uniformly in $t$ for some positive constant $C$ (that depends on $|D^2u|_{0;B_1}$) and some $\iota\in(0,1)$. It remains to show that $D^2u$ is H\"older continuous in time. Differentiating the equation with respect to a unitary direction $e$ we have
\[
F_{ij}(D^2u(x,t))\partial_{ij}(\partial_e u)-\partial_t (\partial_e u)=0.
\]
As before, we obtain for some $\sigma\in(0,1)$
\[
[Du]_{C^{\sigma,\frac{\sigma}{2}}(Q_\frac12)}\leq C|Du|_{0;Q_1}.
\]
This implies that
\[
|Du(x,t_1)-Du(x,t_2)|\leq C|t_1-t_2|^{\frac{\sigma}{2}}
\]
and for fixed $t$ by the elliptic result we have
\[
|D^2u(x_1,t)-D^2u(x_2,t)|\leq C|x_1-x_2|^\iota,
\]
where the constant here depends in particular on $\|u\|_{C^{1,1}}$ and $|\partial_tu|_{0;Q_1}$, so on the full $C^{2,1}$ norm of $u$.
Using Lemma 3.1 p. 78 in \cite{LSU} applied to $Du$ we get that
\[
|D^2u(x,t_1)-D^2u(x,t_2)|\leq C|t_1-t_2|^{\frac{\sigma\iota}{2(1+\iota)}},
\]
which concludes the proof because of the equivalence (see e.g. p. 120 of \cite{KrylovBookHolder}) between the H\"older seminorm $[u]_{C^{\alpha,\beta}(Q_1)}$ defined with respect to the parabolic distance and
\[
[u]'_{C^{\alpha,\beta}(Q_1)}=\sup_{\substack{x,y\in B_1,\\ x\neq y}}\frac{|u(x,t)-u(y,t)|}{|x-y|^\alpha}+\sup_{\substack{t,s\in(-1,0),\\ t\neq s}}\frac{|u(x,t)-u(x,s)|}{|t-s|^\beta}.
\]
\end{proof}
\begin{rem}
The regularity requirement $u\in C^{4,2}$ can be weakened to $C^{2,1}$ using incremental quotients, see for instance \cite{FRRO} or Lemma 14.6 in \cite{Lieberman}.
\end{rem}
One can also reach $C^{\infty}$ regularity by classical Schauder estimates if $F$ is smooth and $u\in C^{2+\alpha,1+\frac{\alpha}{2}}$. Note that no concavity-type assumptions in the matrix variable are required to perform this step.
\begin{prop}\label{C2Cinfty}
Let $\alpha\in(0,1)$ and $u\in C^{2,1}(\Omega\times(t_1,t_2))$ be a solution of
\[
F(x,t,D^2u)-\partial_t u=f(x,t)\text{ in }\Omega\times(t_1,t_2).
\]
Assume that $F\in C^\infty(\Omega\times(t_1,t_2)\times\mathcal{S}_n)$ is uniformly elliptic and $f\in C^\infty(\Omega\times(t_1,t_2))$. Then $u\in C^\infty(\Omega\times(t_1,t_2))$.
\end{prop}
\begin{proof}
The proof uses maximal $L^q$-regularity and Schauder estimates for linear parabolic equations with continuous or H\"older continuous coefficients, we refer to Lemma 14.11 in \cite{Lieberman} for details.
\end{proof}
We can now conclude the following result which, in the concave case, gives a weaker statement than those in \cite{Wang2}.
\begin{cor}
Let $F$ be quasiconcave and satisfy $F\in C^\infty$. Let $u\in C^{2,1}(Q_1)$ be a solution of $F(D^2u)-\partial_t u=0$ in $Q_1$. Then $u\in C^\infty(Q_1)$ and
\[
\|u\|_{C^{2+\alpha,1+\frac{\alpha}{2}}(Q_\frac12)}\leq C(n,\lambda,\Lambda,\|u\|_{C^{2,1}(Q_1)},|F(0)|).
\]
\end{cor}
\begin{proof}
This is a consequence of Proposition \ref{C2Cinfty} and the previous a priori Evans-Krylov estimates.
\end{proof}
\begin{rem}
The previous estimate can be reduced to
\[
\|u\|_{C^{2+\alpha,1+\frac{\alpha}{2}}(Q_\frac12)}\leq C(n,\lambda,\Lambda,|u|_{0;Q_1},|F(0)|)
\]
with the aid of the a priori $C^{2,1}$ estimates via the Bernstein method, see Corollary \ref{smoothconcave}.
\end{rem}
\section{Regularity estimates for fully nonlinear parabolic equations without concavity assumptions}\label{sec;2dpar}
\subsection{Equations with two space variables}
The same proof of Theorem \ref{EKpar} leads to the following regularity estimate for fully nonlinear parabolic equations in dimension 2+1 without any assumption on the operator other than the uniform parabolicity, see Remark \ref{2d} for more details.
\begin{thm}\label{2dpar}
Let $u:B_1\times(-1,0]\to\R$, $B_1\subset\R^2$, and $u$ be a viscosity solution to
\[
F(D^2u)-\partial_t u=0\text{ in }Q_1.
\] 
Assume that $F$ is uniformly parabolic (no other assumptions are required). Then for some $\alpha\in(0,1)$ we have the regularity estimate
\[
\|\partial_t u\|_{C^{\alpha,\alpha/2}(Q_\frac12)}+\|D^2u\|_{C^{\alpha,\alpha/2}(Q_\frac12)}\leq C|u|_{0;Q_1},
\]
where $C$ depends on $\lambda,\Lambda$.
\end{thm}
\begin{proof}
This follows combining the elliptic result by Nirenberg \cite{Nirenberg} (stated for viscosity solutions) with a perturbation argument of L. Caffarelli, cf. Section 8.1 in \cite{CC}, to show the Schauder estimate for the equation
\[
F(D^2u)=f(x)\text{ in }B_1.
\]
with $f\in C^\alpha$. See Theorem \ref{ell2D} for some details of the proof. This can be done once we know that $\partial_t u\in C^{\alpha,\alpha/2}$ for some $\alpha>0$. If we now freeze the time variable, viewing $\partial_t u$ as the right-hand side of the above fully nonlinear elliptic equation, and using that $f(x)=\partial_t u(\cdot,t)\in C^{\alpha}(B_\frac34)$ by \cite{Wang2}, one has that $D^2u(\cdot,t)$ exists in classical sense and is H\"older continuous in $B_\frac12$.
Therefore, in view of the interpolation argument in Theorem \ref{EKpar} and the space-time H\"older continuity of $Du$ obtained in \cite{Wang2} for viscosity solutions, one has the time-H\"older continuity of $D^2u$, which concludes the proof.
\end{proof}
\begin{rem}\label{2d}
The a priori estimate for the fully nonlinear elliptic equation $F(D^2u)=0$ in $B_1\subset\R^2$ without neither concavity nor convexity hypotheses is due to L. Nirenberg \cite{Nirenberg}, see also Theorems 4.9 and 4.26 in \cite{FRRO}. This argument consists in writing the equation satisfied by $w=u_{ee}$ as an equation in divergence form with measurable coefficients of the form $\mathrm{div}(A(x)Dw)=0$ in $B_1$. Then by the De Giorgi-Nash result $w\in C^\alpha(B_\frac12)$. This is stated for viscosity solutions in \cite{SilvestreAnnali}: since we were not able to find a proof in the literature, we decided to provide it in the appendix in Theorem \ref{ell2D}.\\
Another proof was obtained in Remark 2 of \cite{CaffarelliYuan} using a different argument and assuming $F\in C^2$: it relies on observing that $w$ is a subsolution, for a suitable $K>0$, to the equation
\[
F_{ij}\partial_{ij}(e^{Kw})=0.
\]
This idea in \cite{CaffarelliYuan} is the counterpart, in the concave case, of the fact that $u_{ee}$ is a subsolution to a certain nondivergence equation, but it has the advantage to remove the concavity hypothesis.
 In the parabolic setting it was first studied in \cite{Andrews} and recently stated for viscosity solutions in Proposition 1.5 of \cite{SilvestreAnnali}. The above proof slightly shorten \cite{Andrews}, though our estimate is stated for simpler equations independent of $x,t,Du$.
\end{rem}
\begin{rem}
It is worth remarking that fully nonlinear parabolic equations in one-space variables were treated in \cite{Kruzkov}, see also Section XIV.7 of \cite{Lieberman}, without concavity assumptions in the second derivatives.
\end{rem}
\subsection{Equations satisfying Cordes-type conditions}\label{sec;Cordes}

We conclude with a result that does not impose any concavity condition on the operator, but rather a restriction on the ratio among the ellipticity constants, meaning that the fully nonlinear equation is close to be a linear PDE in nondivergence form. Notably, no dimensional restrictions are required as in the previous section. Such a result appeared in \cite{Niu} in the form of a pointwise estimate for stationary equations and in \cite{BW} under more restrictive regularity hypotheses on $F$, while the earlier paper \cite{HuangAIHP} proved an interior estimate that we use here to derive its parabolic analogue. We refer also to the statement of Theorem 6.5 in \cite{SilvestreNotes}: the next theorem extends it to time-dependent equations.
\begin{thm}\label{cordes}
Let $F$ be continuous and uniformly parabolic, and suppose that
\[
\frac{\Lambda}{\lambda}<1+\frac{1}{80n}2^{-\frac72(n+2)^2}.
\]
 Let $u\in C(Q_1)$ be a viscosity solution to $F(D^2u)-\partial_t u=0$. Then $u\in C^{2+\alpha,1+\frac{\alpha}{2}}(\overline{Q}_\frac12)$ for some $\alpha>0$ and
\[
\|u\|_{C^{2+\alpha,1+\frac{\alpha}{2}}(\overline{Q}_\frac12)}\leq C,
\]
where $\alpha\in(0,1)$ and $C$ are universal constants, with $C$ depending also on $|u|_{0;Q_1}$ and $|F(0)|$.
\end{thm}
\begin{proof}
The proof is the same as that in Theorem \ref{EKpar} and consists in regarding the parabolic equation as an elliptic one, exploiting again that $\partial_t u\in C^{\alpha,\frac{\alpha}{2}}$ along with the elliptic result in Proposition 6.1 and Example 6.2 of \cite{HuangAIHP} valid under the Cordes-type condition $\frac{\Lambda}{\lambda}<1+\frac{1}{80n}2^{-\frac72(n+2)^2}$.
\end{proof}
\begin{rem}
As noted at p.17 of \cite{SilvestreNotes}, the counterexamples in \cite{NTV} (see in particular \cite{NTVadv}) hold when $\Lambda/\lambda$ is large, and thus Theorem \ref{cordes} does not contradict the results in \cite{NTVadv,NTV,NV}.
\end{rem}

\section{$C^{2+\alpha,1+\alpha/2}$ regularity estimates for parabolic quasiconcave equations}\label{sec;regpar}
We prove the following parabolic version of Theorem \ref{mainell} for the equation
\begin{equation}\label{parmodel}
F(D^2u)-\partial_t u=0\text{ in }Q_1=B_1\times (-1,0].
\end{equation}
\begin{thm}\label{mainpar}
Let $F$ be continuous, quasiconcave and uniformly parabolic. Let also $u\in C(Q_1)$ be a viscosity solution to \eqref{parmodel}. Then $u\in C^{2+\alpha,1+\frac{\alpha}{2}}(\overline{Q}_\frac12)$ for some $\alpha>0$ and
\[
\|u\|_{C^{2+\alpha,1+\frac{\alpha}{2}}(\overline{Q}_\frac12)}\leq C,
\]
where $\alpha\in(0,1)$ and $C$ are universal constants, with $C$ depending also on $|u|_{0;Q_1}$ and $|F(0)|$.
\end{thm}
The proof makes use of the same idea of the a priori estimate proved in Theorem \ref{EKpar}, namely, being $\partial_t u\in C^{\alpha,\alpha/2}$ for some small $\alpha$ for any uniformly parabolic equation (without any concavity assumption), we regard the equation as an elliptic equation with a H\"older continuous source. In particular, we use the following
\begin{thm}
Let $u\in C(Q_1)$ be a viscosity solution to \eqref{parmodel}  with $F$ uniformly parabolic. Then, for some $\alpha>0$ small depending on $n,\lambda,\Lambda$ we have $u\in C^{1+\alpha,\frac{1+\alpha}{2}}(Q_\frac12)$ and it holds the regularity estimate
\[
\|\partial_t u\|_{C^{\alpha,\frac{\alpha}{2}}(Q_\frac12)}+\|Du\|_{C^{\alpha,\frac{\alpha}{2}}(Q_\frac12)}\leq C(|u|_{0;Q_1}+|F(0)|),
\]
where the constant $C$ is universal.
\end{thm}  
\begin{proof}
The proof can be found in Theorems 4.8 and 4.9 of \cite{Wang2}, and it is based on the fact that the first-order difference quotients (in space to estimate $Du$, and in time to estimate $\partial_t u$) belong to the class $\mathcal{S}(\lambda,\Lambda,0)=\underline{\mathcal{S}}(\lambda,\Lambda,0)\cap \overline{\mathcal{S}}(\lambda,\Lambda,0)$. Thus one can apply Theorems 4.15 and 4.16 of \cite{Wang1} to deduce the desired regularity.
\end{proof}
Once the elliptic result and $\partial_t u\in C^{\alpha,\frac{\alpha}{2}}$  are known, we conclude the following
\begin{prop}
Let $u\in C(Q_1)$ be a viscosity solution to \eqref{parmodel}  with $F$ uniformly parabolic and quasiconcave. Then $D^2u(\cdot,t)$ exists and is H\"older continuous for any fixed value of $t$. Moreover, we have
\[
\|D^2u(\cdot,t)\|_{C^{\alpha}(B_\frac34)}\leq C(n,\lambda,\Lambda,|u|_{0;Q_1},|F(0)|).
\]
\end{prop}
\begin{proof}
Since $\partial_t u$ exists in classical sense and belongs to $C^{\alpha,\frac{\alpha}{2}}$, one can freeze time and regard the equation as an elliptic equation with H\"older continuous right-hand side. In view of the elliptic result we have the existence and the interior $C^{\alpha}$ regularity of $D^2u(\cdot,t)$.
\end{proof}
Finally, we are only left to prove the time-H\"older continuity of $D^2u$. To this aim, we proceed as in \cite{Andrews} and \cite{SilvestreAnnali} to prove the main result
\begin{proof}[Proof of Theorem \ref{mainpar}]
The proof is the same of Proposition 1.5 in \cite{SilvestreAnnali}, see also Sections 3.3 and 3.4 in \cite{Andrews}. Alternatively, one can exploit the interpolation argument used to prove the parabolic a priori estimate.
\end{proof}
\begin{rem}
One can also state a parabolic version of the perturbation results in Section \ref{sec;pert} using those of \cite{Wang1,Wang2} as well as an application to fully nonlinear parabolic PDEs with lower order terms by means of the $C^{1+\alpha,\frac{1+\alpha}{2}}$ results in \cite{Wang2}.
\end{rem}

\section{Other estimates for fully nonlinear parabolic equations concave at infinity}
We prove here some $C^{2,1}$ estimates for fully nonlinear parabolic equations under some (different) concavity assumptions at infinity. These extend a result by L. Caffarelli and Y. Yuan \cite{CaffarelliYuan} to parabolic equations. 
\begin{thm}\label{C11infinity}
Let $F$ be concave (respectively convex) at infinity, or close to a linear function at infinity, i.e. it satisfies one of the following conditions
\begin{itemize}
\item[(i)] $F_{ij,kl}(M)N_{ij}N_{kl}\leq 0$ for $N\in\mathcal{S}_n$ (resp. $\geq0$) when $\|M\|\geq\nu>0$;
\item[(ii)]  $\|D^2F(M)\|\leq \frac{\delta}{\|M\|}$ when $\|M\|\geq \nu$, $\nu>0$ and a small $\delta>0$.
\end{itemize}
Let also $F$ be uniformly parabolic, $F\in C^2$. Let $u\in C^{4,2}(\overline{Q}_1)$ be such that $F(D^2u)-\partial_t u=0$ in $Q_1$. Then
\[
\|u\|_{C^{2,1}(Q_{1/2})}\leq C(\mathrm{osc}_{Q_1}u+|F(0)|+\nu)
\]
Here, $C$ is a universal constant, depending only on $n,\lambda,\Lambda$. 
\end{thm}
\begin{proof}
We follow \cite{Landis}, see also \cite{Kovats}. Take $\chi\in C^2(\overline{Q}_1)$ such that $0\leq\chi\leq 1$ in $\overline{Q}_\frac34$, $\chi=1$ in $Q_\frac12$ and $\chi\equiv0$ on $\partial_{\mathrm{par}}Q_\frac34$ satisfying
\[
0\leq\frac{|D\chi|^2}{\chi}+\|D^2\chi\|+|\partial_t\chi|\leq C
\]
where $C$ depends only on the dimension $n$.\\
 We first discuss the case (i), assuming without loss of generality that $F$ is concave when $\|D^2u\|$ is large. We consider in $A:=Q_\frac34\cap \{x\in Q_\frac34:u_{ee}>\nu\}$ the function
\[
h=\chi^2(u_{ee})^2+\mu u_e^2.
\]
We have
\begin{equation}\label{Pue}
F_{ij}(D^2u)\partial_{ij}u_e-\partial_t u_e=0
\end{equation}
and
\begin{equation}\label{Puee}
\mathcal{P}u_{ee}:=F_{ij}(D^2u)\partial_{ij}u_{ee}-\partial_t u_{ee}=-F_{ij,kl}(D^2u)\partial_{ij}u_e\partial_{kl}u_e.
\end{equation}
Since $F$ is concave when $\|D^2u\|$ is large, the right-hand side in \eqref{Puee} is nonnegative, so
\[
F_{ij}(D^2u)\partial_{ij}u_{ee}-\partial_t u_{ee}\geq0\text{ when }\|D^2u\|\geq\nu.
\]
We drop the dependence on $D^2u$ on $F$ and compute $\mathcal{P}h=F_{ij}\partial_{ij}h-\partial_t h$, using that
\[
0\leq\chi\leq1;
\]
\[
\mathcal{P}u_e=0,\ \mathcal{P}u_{ee}\geq 0\text{ when }\|D^2u\|\geq\nu;
\]
\[
u_{ee}=Du_e\cdot e\leq |Du_e|.
\]
We have
\begin{align*}
\mathcal{P}h&=2u_{ee}^2F_{ij}\partial_i\chi\partial_j\chi+2\chi u_{ee}^2\mathcal{P}\chi+8\chi u_{ee}F_{ij}\partial_i\chi\partial_j u_{ee}+2\chi^2F_{ij}\partial_j u_{ee}\partial_i u_{ee}\\
&+2\chi^2u_{ee}\mathcal{P}u_{ee}+2\mu\mathcal{P}u_e+2\mu F_{ij}\partial_i u_e\partial_j u_e\\
&\geq 2u_{ee}^2F_{ij}\partial_i\chi\partial_j\chi+2\chi u_{ee}^2\mathcal{P}\chi+8\chi u_{ee}F_{ij}\partial_i\chi\partial_j u_{ee}+2\chi^2F_{ij}\partial_j u_{ee}\partial_i u_{ee}+2\mu F_{ij}\partial_i u_e\partial_j u_e\\
&\geq 2\lambda u_{ee}^2|D\chi|^2+2\chi u_{ee}^2\mathcal{P}\chi-8\Lambda\chi u_{ee}|D\chi||Du_{ee}|+2\chi^2\lambda|Du_{ee}|^2+2\mu\lambda|Du_e|^2\\
&\geq 2\lambda u_{ee}^2|D\chi|^2+2\chi u_{ee}^2(\mathcal{P}\chi+\mu\lambda)-8\Lambda\chi u_{ee}|D\chi||Du_{ee}|+2\chi^2\lambda|Du_{ee}|^2.
\end{align*}
We now take $\mu$ large so that
\[
|\mathcal{P}\chi|\leq \frac{\mu\lambda}{2}.
\]
Indeed, this is possible since
\[
|\mathcal{P}\chi|\leq \|DF\|\|D^2\chi\|+|\partial_t\chi|\leq C\Lambda\sqrt{n}+C=C(\Lambda\sqrt{n}+1)
\]
by taking
\[
\mu\geq \frac{2C}{\lambda}(\Lambda\sqrt{n}+1).
\]
This implies, using also the weighted Young's inequality,
\begin{align*}
\mathcal{P}h&\geq 2\lambda u_{ee}^2|D\chi|^2+\chi u_{ee}^2\mu\lambda-8\Lambda\chi u_{ee}|D\chi||Du_{ee}|+2\chi^2\lambda|Du_{ee}|^2\\
&\geq 2\lambda u_{ee}^2|D\chi|^2+\chi u_{ee}^2\mu\lambda-8\Lambda\left(\eps u_{ee}^2|D\chi|^2+\frac{\chi^2|Du_{ee}|^2}{\epsilon}\right)+2\chi^2\lambda|Du_{ee}|^2\\
&\geq (2\lambda-8\Lambda\eps) u_{ee}^2|D\chi|^2+u_{ee}^2\mu\lambda\frac{|D\chi|^2}{C}+\chi^2|Du_{ee}|^2\left(2\lambda-\frac{8\Lambda}{\eps}\right)\\
&=u_{ee}^2|D\chi|^2\left(2\lambda-8\Lambda\eps+\frac{\mu\lambda}{C}\right)+\chi^2|Du_{ee}|^2\left(2\lambda-\frac{8\Lambda}{\eps}\right)\\
&\geq u_{ee}^2|D\chi|^2\left(\frac{\mu\lambda}{C}-8\Lambda\eps\right)+\chi^2|Du_{ee}|^2\left(2\lambda-\frac{8\Lambda}{\eps}\right).
\end{align*}
We first choose $\eps$ so that
\[
\frac{\mu\lambda}{C}-8\Lambda\eps=\frac{\mu\lambda}{2C}\implies \eps=\frac{\mu\lambda}{16\Lambda C}.
\]
Choosing $\mu$ large so that
\[
2\lambda-\frac{8\Lambda}{\eps}=2\lambda-\frac{128 \Lambda^2 C}{\mu\lambda}\geq0,
\]
and hence we can choose $\mu$ once for all such that
\[
\mu\geq\max\left\{\frac{2C}{\lambda}(\Lambda\sqrt{n}+1),\frac{64\Lambda^2 C}{\lambda^2}\right\},
\]
to have $\mathcal{P}h\geq0$ on the set where $u_{ee}>\nu$. 
%On the set $Q_\frac34\cap \{x\in Q_\frac34:u_{ee}\leq0\}$ we have $h=\mu u_e^2$, so, using that $\mathcal{P}u_e=0$ one concludes $\mathcal{P}h=\mathcal{P}(\mu u_e^2)\geq 2\mu \lambda|Du_e|^2\geq0$. Therefore on $Q_\frac34\cap \{x\in Q_\frac34:u_{ee}\leq0\}$ we have $\mathcal{P}h\geq0$. It remains to discuss the case $Q_\frac34\cap \{x\in Q_\frac34:0<u_{ee}\leq \nu\}$. The previous computations show that
%\[
%\mathcal{P}h\geq u_{ee}^2|D\chi|^2\left(\frac{\mu\lambda}{C}-8\Lambda\eps\right)\geq -\nu^2\chi \frac{\mu\lambda}{2C}\geq -\nu^2 \frac{\mu\lambda}{2C}
%\]
Therefore, we conclude that on $A$ the auxiliary function $h$ is a solution of the following differential inequality in nondivergence form
\[
\mathcal{P}h\geq 0.
\]
We can thus apply the maximum principle on the set $A$, observing that
\[
A_\frac34=(B_\frac34\cap \{x\in B_\frac34:u_{ee}(t)>\nu\})\times\left(-\frac{9}{16},0\right):=W\times\left(-\frac{9}{16},0\right)
\]
and
\[
\partial W\subseteq \partial B_\frac34\cup \{x\in B_\frac34:u_{ee}(t)=\nu\}.
\]
Therefore
\begin{align*}
\partial_{\mathrm{par}} A_\frac34&\equiv\left[\partial W\times\left(-\frac{9}{16},0\right)\right]\cup \left[W\times\left\{t=-\frac{9}{16}\right\}\right]\\
&\subseteq \left[(\partial B_\frac34\cup \{x\in B_\frac34:u_{ee}(t)=\nu\})\times \left(-\frac{9}{16},0\right)\right]\cup \left[W\times\{t=-\frac{9}{16}\}\right]\\
&\equiv \left[\partial B_\frac34 \times \left(-\frac{9}{16},0\right)\right]\cup \{(x,t)\in Q_\frac34:u_{ee}(t)=\nu\})\cup \left(B_\frac34\times\left\{t=-\frac{9}{16}\right\}\right)\\
&\cup \left\{(x,t)\in B_\frac34\times\left\{t=-\frac{9}{16}\right\}:u_{ee}(t)>\nu\right\}
\end{align*}
and by the maximum principle (using that $\chi\equiv0$ on $\partial_{\mathrm{par}}Q_\frac34$)
\[
\sup_{A} h\leq \sup_{\partial_{\mathrm{par}} A}h\leq \mu\sup_{Q_\frac34}(u_e)^2+\nu^2.
\]
Moreover
\[
\sup_{A} h\geq \sup_{ Q_\frac12}(u_{ee}^+)^2.
\]
A similar calculation for the function $z=\delta(M-u)^2+\eta^2|Du|^2$ (or alternatively the estimate for viscosity solutions in \cite{Wang2}), gives the estimate
\[
\sup_{ Q_\frac12}(u_{ee}^+)\leq \tilde{C}\left(\mu\sup_{Q_1}|u|+\nu\right).
\]
The estimate on the full Hessian follows by the equation, using the uniform ellipticity, Lemma \ref{twoside} and the estimate on $|\partial_t u|_{0;Q_1}$ from \cite{Wang2}, see also \cite{GutierrezHuang}. Indeed, the same arguments of Lemma 6.4 in \cite{CC} show that when $F(D^2u)-\partial_t u=0$, then
\[
F(0)\leq F(D^2u)+\Lambda\|(D^2u)^+\|-\lambda\|(D^2u)^-\|\leq |\partial_tu|_{0;Q_1}+(\Lambda+\lambda)\|(D^2u)^+\|-\lambda\|D^2u\|
\]
This implies
\[
\|D^2u\|\leq \frac{\lambda+\Lambda}{\lambda}\|(D^2u)^+\|+\frac{1}{\lambda}|\partial_t u|_{0;Q_1}+\frac{|F(0)|}{\lambda}
\]
Moreover, we have
\[
\|(D^2u)^+\|=\sup_{e\in\R^n,|e|=1}(D^2ue\cdot e)^+,
\]
so
\[
\|D^2u\|\leq \frac{\lambda+\Lambda}{\lambda}\cdot\sup_{e\in\R^n,|e|=1}(D^2ue\cdot e)^++\frac{1}{\lambda}|\partial_t u|_{0;Q_1}+\frac{|F(0)|}{\lambda}.
\]
We now consider the case (ii) where $\|D^2F(D^2u)\|\leq \frac{\delta(n,\lambda)}{\|D^2u\|}$ when $\|D^2u\|\geq \nu$, $\nu>0$. We consider the function $h$ defined as
\[
h=\mu\sum_{e\in\Gamma}u_e^2+\chi\sum_{e\in\Gamma}(u_{ee}^+)^2,
\]
where
\[
\Gamma=\left\{e_1,...,e_n,\frac{1}{\sqrt{2}}(e_i\pm e_j),1\leq i<j\leq n\right\}.
\]
This can be done in view of Lemma 17.13 of \cite{GT}. We compute again $F_{ij}\partial_{ij}h-\partial_t h$ using now that
\[
\mathcal{P}u_e=0,\ \mathcal{P}u_{ee}=-F_{ij,kl}(D^2u)\partial_{ij}u_e\partial_{kl}u_e.
\]
\begin{align*}
\mathcal{P}h&=\sum_{e\in\Gamma}\left[2u_{ee}^2F_{ij}\partial_i\chi\partial_j\chi+2\chi u_{ee}^2\mathcal{P}\chi+8\chi u_{ee}F_{ij}\partial_i\chi\partial_j u_{ee}+2\chi^2F_{ij}\partial_j u_{ee}\partial_i u_{ee}
+2\chi^2u_{ee}\mathcal{P}u_{ee}\right]\\
&+2\mu\sum_{e\in\Gamma}\left(u_e\mathcal{P}u_e+ F_{ij}\partial_i u_e\partial_j u_e\right)\\
&\geq 2\mu\lambda \sum_{e\in\Gamma}|Du_e|^2+\sum_{e\in\Gamma}\left[-2\chi^2u_{ee}F_{ij,kl}\partial_{ij}u_e\partial_{kl}u_e+2u_{ee}^2F_{ij}\partial_i\chi\partial_j\chi\right.\\&\left.+2\chi u_{ee}^2\mathcal{P}\chi+8\chi u_{ee}F_{ij}\partial_i\chi\partial_j u_{ee}+2\chi^2F_{ij}\partial_j u_{ee}\partial_i u_{ee}  \right]\\
&=2\mu\lambda \sum_{e\in\Gamma}|Du_e|^2+\sum_{e\in\Gamma}\left[u_{ee}^2(2F_{ij}\partial_i\chi\partial_j\chi+2\chi \mathcal{P}\chi)\right]\\
&+\chi^2\sum_{e\in\Gamma}\left[2F_{ij}\partial_j u_{ee}\partial_i u_{ee}-2u_{ee}F_{ij,kl}(D^2u)\partial_{ij}u_e\partial_{kl}u_e\right]-8\Lambda\sum_{e\in\Gamma}\left(\eps u_{ee}^2|D\chi|^2+\frac{\chi^2|Du_{ee}|^2}{\epsilon}\right)\\
&\geq 2\mu\lambda \sum_{e\in\Gamma}|Du_e|^2+\sum_{e\in\Gamma}u_{ee}^2\left[2\lambda|D\chi|^2+2\chi \mathcal{P}\chi \right]\\
&+\chi^2\sum_{e\in\Gamma}\left[2\lambda|Du_{ee}|^2-2\delta\|D^2u_e\|^2\right]-8\Lambda\sum_{e\in\Gamma}\left(\epsilon u_{ee}^2|D\chi|^2+\frac{\chi^2|Du_{ee}|^2}{\epsilon}\right).
\end{align*}
We now choose 
\[
\epsilon=\frac{8\Lambda}{\lambda}
\]
so that the previous term becomes
\begin{align*}
2\mu\lambda \sum_{e\in\Gamma}|Du_e|^2&+\sum_{e\in\Gamma}u_{ee}^2\left[\left(2\lambda-\frac{64\Lambda^2}{\lambda}\right)|D\chi|^2-2\chi |\mathcal{P}\chi| \right]
+\chi^2\sum_{e\in\Gamma}\left[\lambda|Du_{ee}|^2-2\delta\|D^2u_e\|^2\right]\\
&\geq 2\mu\lambda \sum_{e\in\Gamma}|Du_e|^2+\sum_{e\in\Gamma}|Du_e|^2\left[\left(2\lambda-\frac{64\Lambda^2}{\lambda}\right)C-2C \right]\\
&+\chi^2\sum_{e\in\Gamma}\left[\lambda|Du_{ee}|^2-2\delta\|D^2u_e\|^2\right]\\
&= \sum_{e\in\Gamma}|Du_e|^2\left[2\mu\lambda+ \left(2\lambda-2-\frac{64\Lambda^2}{\lambda}\right)C\right]+\chi^2\sum_{e\in\Gamma}\left[\lambda|Du_{ee}|^2-2\delta\|D^2u_e\|^2\right].
\end{align*}
One then chooses $\mu$ large such that
\[
\mu\geq \frac{\left(\frac{64\Lambda^2-2\lambda+2}{\lambda}\right)C}{2\lambda}
\]
and $\delta$ small to conclude $\mathcal{P}h\geq0$ on the set where $\{x\in Q_\frac34:u_{ee}>\nu\}$. By the parabolic maximum principle, arguing as before, we obtain
\[
\sup_{A_\frac34}h\leq \mu n^2\sup_{Q_\frac34}|Du|^2+n^2\nu^2.
\]
Again, the full two-side estimate on the Hessian $|D^2u|_{0;Q_\frac12}$ follows using the uniform ellipticity and the estimate on $|\partial_t u|_{0;Q_1}$ along with the estimate for $\sup_{Q_\frac34}|Du|^2$ both in terms of $|u|_{0;Q_1}$, which are valid without concavity conditions. We then conclude
\[
|D^2u|_{0;Q_\frac12}\leq C(n,\lambda,\Lambda)(|u|_{0;Q_1}+\nu).
\]
\end{proof}
\begin{rem}
The previous results can be extended to operators $F=F(t,M)$ assuming only measurability in the time-variable, see \cite{Kovats}.
\end{rem}
\begin{rem}
When (i) holds with $\nu=0$, i.e. $F$ is concave, the full a priori estimate in $C^{2,1}$ by the Bernstein method is new. J. Kovats \cite{Kovats} proved a one-side bound when $F=F(t,M)$ is concave in $\mathcal{S}_n$ and measurable in $t$. In that case, it is sufficient to note that on the set where $u_{ee}\leq0$ we have $\mathcal{P}h\geq0$, so $\mathcal{P}h\geq0$ on the whole cylinder, where one can directly apply the maximum principle. 
\end{rem}
\begin{rem}
Examples of (stationary) operators that satisfy the previous assumptions (i)-(ii) or related conditions can be found in \cite{CaffarelliYuan,HuangAIHP}.
\end{rem}
\begin{rem}
A similar condition as (ii) has been used in Theorem 1.3 of \cite{HuangAIHP} to prove Sobolev regularity estimates of viscosity solutions. Moreover, the estimates in Theorem \ref{C11infinity} can be derived for less regular solutions in $C^{2,1}$. It remains an open problem, even in the stationary setting of \cite{CaffarelliYuan}, to prove a regularity estimate valid for viscosity solutions, along with an Evans-Krylov theorem, under the assumptions (i) and (ii) of Theorem \ref{C11infinity}.
\end{rem}
\begin{rem}
Parabolic $C^{1,1}$ estimates can be used to prove $W^{2,1}_p$ estimates for non-homogeneous equations on the line of Theorem \ref{czest}. This can be done by means of the results in \cite{Wang1}.
\end{rem}
In the concave/convex case the previous result combined with the Evans-Krylov estimate give the following weaker version of the $C^{2+\alpha,1+\alpha/2}$ estimate for fully nonlinear concave parabolic equations than \cite{Wang2}:
\begin{cor}\label{smoothconcave}
Let $F$ be concave and $F\in C^\infty$. Let $u\in C^{2,1}(Q_1)$ be a solution to $F(D^2u)-\partial_tu=0$ in $Q_1$. Then $u\in C^\infty$ and
\[
\|u\|_{C^{2+\alpha,1+\alpha/2}(Q_\frac12)}\leq C(|u|_{0;Q_1}+|F(0)|).
\]
\end{cor}
Following \cite{ColdingMinicozzi}, given $\gamma>0$, we let $\mathcal{P}_\gamma(\R^n)$ be the space of ancient solutions to the fully nonlinear parabolic equation $F(D^2u)-\partial_t u=0$ in $Q_R=B_R\times[-R^2,0]$, so that there exists a constant $C_u$ such that
\[
\sup_{Q_R}|u|\leq C_u(1+R)^\gamma.
\]
Then, $\mathcal{H}_\gamma(\R^n)$ is the space of solutions to $F(D^2u)=0$ having growth at most $\gamma\geq0$.
This, together with the Bernstein bound for first derivatives, implies the following Liouville result for ancient solutions to parabolic equations under the assumption that $F$ is either concave, convex or close to a linear function near infinity. The result appears to be new even in the more classical concave/convex case.
\begin{cor}
Let $F$ be uniformly parabolic and $F(0)=0$. Assume that $F$ is either concave/convex or concave/convex/close to a linear function at infinity. If $\gamma<1$, then $\mathcal{P}_\gamma(\R^n)=\{\text{constant functions}\}$. If $\gamma<2$ we have $\mathcal{P}_\gamma(\R^n)=\mathcal{H}_\gamma(\R^n)$.
\end{cor}
\begin{proof}
Consider first the case $\gamma<1$. It is sufficient to rescale the estimate in Theorem \ref{C11infinity} for equations defined in $Q_R$ and obtain
\[
|Du|_{0;Q_{R/2}}\leq \frac{C}{R}(|u|_{0;Q_R}+\nu).
\]
Indeed, if $u$ solves $F(D^2u)-\partial_t u=0$ in $Q_R$, then $\tilde u(x,t)=u(Rx,R^2t)$ solves
\[
\tilde F(D^2\tilde u)-\partial_t\tilde u=0\text{ in }Q_1,\ \tilde F(M)=R^2F\left(\frac{1}{R^2}M\right),
\]
where $\tilde F$ is uniformly elliptic with the same ellipticity constants of $F$. This implies that
\[
R|Du|_{0;Q_{R/2}}=|Du|_{0;Q_{1/2}}\leq C(|u|_{0;Q_1}+\nu)=C(|u|_{0;Q_R}+\nu).
\]
The conclusion follows by sending $R\to\infty$ in the gradient bound, using the growth condition on $u$ to see that $u$ must be constant in space. Therefore $D^2u(x,t)=0$ and, using the equation, $\partial_t u=F(D^2u)=F(0)=0$, so $u$ must also be constant in time.\\
To prove the second statement, it is enough to exploit the time-derivative bound
%\[
%F(D^2u)-F(0)\leq \Lambda \mathrm{Tr}(D^2u)\leq \Lambda\sqrt{N}|D^2u|
%\]
%and so using the equation and second-order estimates we find
\[
\sup_{Q_{R/2}}|\partial_tu|\leq \frac{C}{R^2}\left(|u|_{0;Q_R}+\nu\right).
\]
Therefore, by sending $R\to\infty$, it follows that $u$ is constant in time and the thesis follows.
\end{proof}

\appendix
\section{Regularity estimates for viscosity solutions of fully nonlinear elliptic equations in two variables}

The main aim of this section is to prove the following extension of Nirenberg's result \cite{Nirenberg} for viscosity solutions. The statement is taken from Theorem 2.4 in \cite{SilvestreAnnali}, see also Theorem 6.2 in \cite{SilvestreNotes} and Theorems 4.9 and 4.26 in \cite{FRRO}:
\begin{thm}\label{ell2D}
Let $u:B_1\to\R$, $B_1\subset\R^2$, and $u$ be a continuous viscosity solution to
\[
F(D^2u)=f(x)\text{ in }B_1.
\] 
Assume that $F:\mathcal{S}_2\to\R$ is uniformly elliptic (no other assumptions are required) and $f\in C^\alpha(B_1)$. Then for some small $\alpha\in(0,1)$ we have the regularity estimate
\[
\|u\|_{C^{2,\alpha}(B_\frac12)}\leq C(|u|_{0;B_1}+\|f\|_{C^\alpha(B_1)}),
\]
where $\alpha$ and $C$ depend on $\lambda,\Lambda$.
\end{thm}
\begin{proof}
We start with the homogeneous case $f=0$. In this setting the proof makes use of the a priori $C^{2,\alpha}$ estimate in Theorem 4.9 of \cite{FRRO} up to the boundary \cite{GT}, and the method of continuity to show the existence of $C^{2,\alpha}$ solutions for the Dirichlet problem
\[
\begin{cases}
F(D^2u)=0&\text{ in }B_1,\\
u=g&\text{ on }\partial B_1,
\end{cases}
\]
with any given boundary datum $g\in C(\partial B_1)$. These have been achieved for instance in Section 17.3 of \cite{GT}, Theorems 17.10 and 17.12. Note that having $F:\mathcal{S}_2\to\R$ uniformly elliptic is enough since the uniform ellipticity implies the increasing monotonicity and the Lipschitz continuity in the matrix entry, as pointed out in Theorem 4.9 of \cite{FRRO}. Once the existence of a smooth solution is established, one can conclude that the desired regularity estimate holds for viscosity solutions by means of the uniqueness of solutions, see e.g. Corollary 5.4 in \cite{CC}.\\
To conclude, the estimate in the non-homogeneous case follows by the perturbation argument in Theorem 8.1 of \cite{CC}, which makes use of the regularity estimate for the homogeneous problem.
\end{proof}

%\subsection{Semiconcave functionals with small semiconcavity constant}
%\subsection{Locally concave operators at infinity}

%\bibliography{fullybib}
%\bibliographystyle{abbrv}

\end{document}